\documentclass[a4paper,12pt,reqno,oneside]{amsart}

\usepackage[english]{babel} 
\usepackage[utf8]{inputenc}  
\usepackage[T1]{fontenc} 
\usepackage{amssymb, mathtools}
\usepackage{geometry}

\theoremstyle{plain}
\newtheorem{theorem}{Theorem}[section]
\newtheorem*{theorem*}{Theorem}
\newtheorem{lemma}[theorem]{Lemma}
\newtheorem{prop}[theorem]{Proposition}
\newtheorem{corollary}{Corollary}
\newtheorem{definition}[theorem]{Definition}

\newenvironment{claim}[1]{\par\noindent\underline{\textit{Claim}:}\space#1}{}
\newenvironment{claimproof}[1]{\par\noindent{\textit{Proof of Claim}:}\space#1}{\hfill $\blacksquare$}

\newcommand{\pkld}[2]{\mathcal{P}_{{#1}}({#2})}
\newcommand{\pkl}{\mathcal{P}_{\kappa}\lambda}
\newcommand\scalemath[2]{\scalebox{#1}{\mbox{\ensuremath{\displaystyle #2}}}}
\newcommand{\pq}[1]{\scalemath{0.8}{#1}}
\newcommand{\pk}[1]{\mathcal{P}_{\kappa}({#1})}

\title[HIGHER STATIONARITY AND DERIVED TOPOLOGIES ON $\pk{A}$]{HIGHER STATIONARITY AND DERIVED\\ TOPOLOGIES ON $\pk{A}$}
\author{M. Catalina Torres}
\address{Departament de Matem\`atiques i Inform\`atica, Universitat de Barcelona. 
Gran Via de les Corts Catalanes, 585,
08007 Barcelona, Catalonia.}
\email{mactorrespa@gmail.com}

\begin{document}

\begin{abstract}
Let $\kappa$ be a regular limit cardinal,  $\kappa \subseteq A$. We study a notion of $n$-s-stationarity on $\pk{A}$. We construct a sequence of topologies $\langle \tau_0, \tau_1, \dots  \rangle $  on $\pk{A}$ characterising the simultaneous reflection of a pair of $n$-s-stationary sets in terms of elements in the base of $\tau_n$. This result constitutes a complete generalisation to the context of $\pk{A}$ of Bagaria's prior characterisation of $n$-simultaneous reflection in terms of derived topologies on ordinals.
\end{abstract}

\maketitle

 \section{Introduction}
 
The study of combinatorial properties of $\pkl:= \{ x \subseteq \lambda : |x| < \kappa\}$ where  $\kappa$  denotes  an uncountable regular cardinal and $\kappa \leq \lambda$, has a long history. It originated  considering natural generalisations of combinatorial  properties  on an infinite cardinal $\kappa$ within the context of large cardinal notions (see \cite{J1,J2,Ku}). For example,  strongly compact and supercompact cardinals can be seen as strengthenings  of weakly compact and measurable cardinals, respectively, by extending certain properties from a single cardinal, $\kappa$, to $\pkl$.  This work later lead to the study of various related topics, including ideals, filters, stationary sets, and stationary reflection on $\pkl$ (see, e.g., \cite{J3,Car0,Car1,Do,MFS,Sh2}). In this paper we investigate some generalisations of reflection properties from the case of $\kappa$ to the case of $\pkl$, with the aim of yielding compelling results that potentially exhibit significantly higher levels of consistency strength.\\

In \cite{B0}, Bagaria introduced the following iterated notion of stationary reflection for a given limit ordinal $\alpha$: $A\subseteq \alpha$ is \textit{$0$-stationary} in $\alpha$ if and only if it is unbounded in $\alpha$. For $ \xi >0$, $A \subseteq \alpha$ is $\xi$-stationary in $\alpha$ if and only if for every $\zeta<\xi$ and every $S$ $\zeta$-stationary in $\alpha$, there is a limit ordinal  $\beta < \alpha$  such that $S \cap \beta$ is $\zeta$-stationary in $\beta$. A regular cardinal $\kappa$ is $\xi$-stationary if and only if $\kappa$  is $\xi$-stationary in $\kappa$. It is easy to see that a subset $S\subseteq \alpha$ is $1$-stationary in $\alpha$ if and only if $S$ is stationary in $\alpha$. Thus, an ordinal $\alpha$ is $2$-stationary if and only if it reflects stationary sets. 
 In \cite{B1}, Bagaria, Magidor and Sakai showed that in the constructible universe,L, this iterated form of stationarity was tightly related with the notion of indescribability. They proved that in $L$, a regular cardinal is $(n+1)$-stationary if and only if it is $\Pi_n^1$ indescribable. Later on, in \cite{B2}, Bagaria  revisited $n$-stationarity on ordinals, this time in the stronger form of $n$-simultaneous stationarity. Given a limit ordinal $\alpha$: $A\subseteq \alpha$ is \textit{$0$-simultaneously-stationary} ($0$-s-stationary for short) in $\alpha$ if and only if it is unbounded in $\alpha$. For $ \xi >0$, $A \subseteq \alpha$ is $\xi$-simultaneously-stationary  ($\xi$-s-stationary for short) in $\alpha$ if and only if for every $\zeta<\xi$ and every $S,T$ $\zeta$-stationary in $\alpha$, there is a limit ordinal $\beta < \alpha$  such that $S \cap \beta$ and $T \cap \beta$ are both  $\zeta$-stationary in $\beta$. As shown in  \cite{B2} $\xi$-simultaneous stationarity is characterised in terms of the non-discreteness of derived topologies on ordinals.\\
 
Given a topological space $(X,\tau)$, the Cantor derivative operator is the map $d_{\tau}: \mathcal{P}(X) \rightarrow \mathcal{P}(X)$ such that $$d_{\tau}(A):= \{x \in X : x \text{ is a limit point of } A \text{ in }\tau \}.$$ For an ordinal $\delta$, the set $\mathcal{B}_0:=\{0\} \cup \{ (\alpha, \beta) : \alpha < \beta \leq \delta\}$ constitutes a base for a topology $\tau_0$ in $\delta$, namely the interval topology. In \cite{B2} Bagaria  constructed an increasing sequence of topologies $\langle \tau_0, \tau_1 \dots , \tau_{\xi} \dots \rangle$ on $\delta$ by means of successive  applications of Cantor's derivative operator on the topological space $(\delta , \tau_0)$.  The sets $d_{\xi}(A):=d_{\tau_{\xi}}(A)$ for $A \subseteq \delta$ are taken as new open sets  at  step $\xi+1$, and unions are taken at limit stages. That is,  $\tau_{\xi+1}$ is defined to be the topology generated by  $\mathcal{B}_{\xi+1}:= \mathcal{B}_{\xi}\cup \{d_{\xi}(A): A \subseteq \delta\}$, and $\tau_{\xi}$ is is defined to be the one generated by  $\mathcal{B}_{\xi}:=\bigcup_{\zeta < \xi}  \mathcal{B}_{\zeta}$ when $\xi$ is limit ordinal.\\

Bagaria proved  various properties of $d_{\xi}$ acting on sets of ordinals, particularly on simultaneous stationary sets. He showed that for all $A \subseteq \delta$, $$d_{\xi}(A)=\{\alpha < \delta : A \text{ is } \xi \text{-s-stationary in } \alpha \}.$$ This result led to two of the main results of his paper: \vspace{2mm}

\textit{Theorem 2.11, \cite{B2}. For every $\xi$, an ordinal $\alpha < \delta$ is not isolated in the $\tau_{\xi}$ topology on $\delta$ if and only if $\alpha$ is $\xi$-s-stationary in $\alpha$.} \vspace{2mm}

\textit{Theorem 3.1, \cite{B2}. For every $\xi$, an ordinal $\alpha$  is $\xi$-s-stationary in $\alpha$ if and only if  $\mathcal{I}_{\alpha}^{\xi}=\{ x \subseteq \alpha : x$ is not s-stationary in $\alpha \}$ is a proper ideal.}\\

The investigation into higher stationarity on ordinals of \cite{B0,B1,B2} prompted an inaugural effort to define higher stationarity on $\pk{A}:=\{x \subseteq A : |x| < \kappa\}$. In  \cite{S1}, H. Brickhill, S. Fuchino and H. Sakai proposed a definition of $n$-stationarity in $P_{\kappa}(A)$, where $\kappa$ is a regular limit cardinal and $\kappa \subseteq A$. Additionally, they suggested  that the concept of indescribability  in $\pk{A}$, given by Baumgartner in \cite{Bau}, may yield the equivalence (in $L$) between indescribability and high stationary reflection, similarly as in the ordinal case (\cite{B1},\cite{B2}). In  \cite{S1}, the authors also suggested  an investigation  into the consistency strength of $n$-stationarity in $\pk{A}$. Specifically, they asserted that if $\kappa$ is $\lambda$-supercompact, then $\pkl$ is $n$-stationary for all $n < \omega$, leaving its converse as an open question.\\

In this paper, our focus lies on exploring the definition of $n$-stationarity proposed in \cite{S1}, with the aim of extending the results obtained in \cite{B2} concerning derived topologies on ordinals to the context of $\pk{A}$. Specifically, we examine an equivalent version of this notion and establish that $\kappa$ being weakly Mahlo is both a sufficient and necessary condition for demanding the $1$-stationarity of $\pk{A}$. Furthermore, we observe that a similar condition for $2$-stationarity may require a severe strengthening of the large cardinal properties of $\kappa$, illustrating the first analogy with the ordinal case. Our exploration reveals that $n$-stationary sets in $\pk{A}$  exhibit connections among themselves and with club, stationary, and unbounded sets, mirroring the connections observed in the ordinal case. The only apparent exception is that a stationary subset of $\pk{A}$ may not necessarily be $1$-stationary, which, as we will later emphasise, does not constitute a significant obstacle to our purpose. Nevertheless, we did found that $1$-stationarity aligns with another form of stationarity, strong stationarity.\\
 
We extend the notion of simultaneous stationarity to the case of  $\pk{A}$, by defining, for each $T \subseteq \pk{A}$, the set $d_n^s(T)$  of  points in $\pk{A}$ in which $T$ is $n$-simultaneously-stationary. We prove that these sets behave exactly like the derived sets in the ordinal case. Particularly, we establish an analogue of  Proposition 2.10 in  \cite{B2}, which served as the key to proving one of Bagaria's main results in \cite{B2}. This exploration sheds light on how to define a starting topology for $\pk{A}$. Then we define a sequence of topologies on  $\pk{A}$ for which the sets $d_n^s(T)$ coincide with the actual derived sets on $\pk{A}$. These results are summarised in Theorems \ref{th2} and \ref{t3}, which generalise Theorems 2.11 and 3.1 in \cite{B2} to the case of $\pk{A}$. \\

In the last section, we present a sufficient condition for $\pk{\kappa}$ to be $n$-stationary and $n$-simultaneously-stationary for all $n < \omega$, namely that $\kappa$ is totally indescribable (Corollary \ref{totind}). Recall that in the ordinal case,  $\kappa$ totally indescribable implies that $\kappa$ is $n$-simultaneously stationary on $\kappa$ for all $n$. Thus, our proof on $\pk{\kappa}$, serves as corroboration that, in the context of high stationarity, $\pk{A}$ is as well a generalisation of $\kappa$, identifying  $\pk{\kappa}$ and $\kappa$. Finally, we provide a sufficient condition for $\pkl$ to be $n$-simultaneously-stationary for all $n < \omega$, namely  that $\kappa$ is $\lambda$-supercompact (Theorem \ref{lsc}). This also yields a proof for the analogous (non-simultaneous) assertion  made in  \cite{S1} concerning  to $n$-stationary sets (Corollary \ref{lsc2}).  \\
 
 We want to acknowledge that, simultaneously and independently, Cody, Lambie-Hanson and Zhang \cite{BCJ} defined a sequence of two-cardinal derived topologies and obtained results similar to those in Section 3 and 4 of the present article involving the relationship between various two-cardinal notions of $\xi$-s-stationarity and two-cardinal derived topologies.

 \section{Notation and framework }
 
Everywhere in the sequel, if not specified otherwise, $\kappa$  denotes  a  regular limit cardinal, and $A$ denotes a set such that $\kappa \subseteq A$. Recall that $\pk{A}$ denotes the set $\{ x \subseteq A : |x| < \kappa\}$. In \cite{J1,J2,J3}, Jech defined the following  well known notions for a set $S\subseteq \pk{A}$; 

\begin{itemize}
\item[(1)] $S$ is \textit{unbounded ($\subseteq$-cofinal)} in $\pk{A}$ iff for any $x \in \pk{A}$ there is some $y \in S$ such that $x \subseteq y$. 
\item[(2)] $S$ is \textit{closed} in $\pk{A}$ iff for any $\{ x_{\xi}  : \xi < \beta \} \subseteq S$ with $\beta < \kappa$ and $x_{\xi} \subseteq x_{\zeta}$ for $\xi \leq \zeta < \beta$, $\bigcup_{\xi < \beta} x_{\xi} \in S$. 
\item[(3)] $S$ is \textit{club} in $\pk{A}$ iff $S$ is closed and cofinal in  $\pk{A}$. \item[(4)] $S$ is \textit{stationary} in $\pk{A}$ iff for any $C$ club in $\pk{A}$, $S \cap C \neq \varnothing$. \\
\end{itemize}

Originally, Jech employed the term $``$unbounded$"$ for definition (1), however, in the present work, we opt for the terminology $S$ is $``$$\subseteq$-cofinal$"$ or simply $S$ is $``$cofinal$"$, to avoid confusion with  the property that for any $x \in \pk{A}$ there is some $y \in S$ such that $y \not\subseteq x$. While in the ordinal case these properties happen to be equivalent, in the context of $\pk{A}$ -due to a lack of total order- we do not have this correspondence.\\

The following are some well known facts that can be found easily in the literature \cite{J3,J4, PV}. For every $x \in \pk{A}$, the sets  $$x{\uparrow} \mathrel{:=} \{y \in \pk{A} : x \subseteq y\}  \ \text{ and } \ \mathcal{C}_{\kappa}:=\{ x \in \pk{A} : x \cap \kappa \text{ is a  cardinal}\}$$ are club subsets of $\pk{A}$, recall that for $C_{\kappa}$ to be a club we need that $\kappa$ is a limit cardinal.  A set $T$ is called directed if for any $x,y \in T$, there is a $z \in T$ such that $x\cup y \subseteq z$. A subset $C$ of $\pk{A}$ is  closed if and only if for every directed set $T\subseteq C$ of cardinality $< \kappa$, $\bigcup T \in C $ (See \cite{J3}). A set $S\subseteq \kappa$ is unbounded (or closed, or stationary) in the sense of $\kappa$ if and only if it is cofinal (or closed, or stationary) in the sense of $\pk{\kappa}$. Club subsets of $\pk{A}$ generate a filter,  the \textit{club filter on $\pk{A}$} and we denote its corresponding dual ideal by $NS_{\kappa,A}$. \\

Finally, recall that, given a topological space $\langle X, \tau \rangle$,  $x \in X$ and $S \subseteq X$, $x$ is  a limit point of $S$ in the topology $\tau$  if and only if for every open set $U\in \tau$ such that $x\in U$, $U \cap (S \setminus \{x \}) \neq \varnothing$. Also, $\mathcal{B}\subseteq \mathcal{P}(X)$ is a base for a topology in $X$ if and only if (1) $X = \bigcup \mathcal{B}$, and (2) if $x \in V_1 \cap V_2$ for $V_1,V_2 \in \mathcal{B}$, then $x \in V$ for some $V \in  \mathcal{B}$, such that $V \subseteq V_1 \cap V_2$. \\

As mentioned in the introduction, Bagaria  \cite{B2}, employed the symbol "$d$" to represent the Cantor derivative operator, which extracts from a given set the set of its limit points in a specific topology. However, for the sake of our presentation, we adopt the symbol  "$\partial$" for this operation. More precisely, for a topological space $(X, \tau)$ and a subset $A \subseteq X$, we define  the operator $\partial_{\tau} : \mathcal{P}(X) \rightarrow \mathcal{P}(X)$ by 
\begin{align*}
\partial_{\tau}(T):= \{x \in T : x \text{ is a limit point of } T \text{  in the topology }\tau \}.
\end{align*} 
This choice of notation allows us to reserve the symbol "$d$" for denoting the set of points for where the higher stationary of some $T \subseteq \pk{A}$ is reflected. Ultimately, our goal is to equate this set with $\partial_{\tau}(T)$.

 \section{Higher stationary subsets of $\pk{A}$ }

Bagaria's definition of higher stationarity can be understood as an iterative process, akin to the reflection of certain sets. It happens to be precisely an iteration of the notion of stationarity in ordinals, because the following equivalence holds: $S$ is a stationary in an ordinal $\alpha$ of uncountable cofinality if and only if for all unbounded subset $T$ of $\alpha$, there is some $\beta \in S$ such that $T \cap \beta$ is unbounded in $\beta$. This is, $S$ is a stationary in $\alpha$ if and only if $S$ is $1$-stationary in $\alpha$. Similarly, the notion of $2$-stationarity in $\alpha$ aligns with the notion of reflection of stationary sets in $\alpha$, which has been extensively studied. However, the implications of higher stationarity, as elucidated in \cite{B2},  primarily emerges from a novel approach: the iterative reflection of specific sets. \\

Reflection of stationary sets in $\pk{A}$, where $A$ is an ordinal, has also been extensively studied by authors such as Jech, Sakai, and Shelah \cite{Sh1,Sh2,MFS,S2,S3}. Since any form of reflection in $\pk{A}$ involves at least two parameters, namely $\kappa$ and $A$, several definitions have been proposed. The proposal of extending the $``$iterative process of reflecting certain sets$"$ made in  \cite{B2}, to the context of $\pk{A}$ is, however, quite novel. This approach was recently introduced by  H. Brickhill, S. Fuchino, and H. Sakai, as presented in \cite{S1}. This particular concept will be our focus, as it closely relates to the notion from \cite{B2} that we aim to generalise.  \vspace{2mm}

\begin{definition}[H. Brickhill, S. Fuchino and H. Sakai \cite{S1}]\label{nst1} 
Let $\kappa$ be a regular cardinal and let $A$ be a set with $\kappa \subseteq A$.
\begin{enumerate}
\item $S\subseteq \pk{A}$ is \textbf{$0$-stationary} (BFS) in $\pk{A}$ iff  $S$  is cofinal in $\pk{A}$.
\item Let $0<n<\omega$. $S\subseteq \pk{A}$ is \textbf{$n$-stationary} (BFS) in $\pk{A}$ iff for all $m<n$ and  all $T \subseteq \pk{A}$ $m$-stationary in $\pk{A}$, there is $x \in S$ such that 
\begin{itemize}
\item[(i)] $\mu := x \cap \kappa$ is a regular cardinal.
\item[(ii)] $T\cap \pkld{\mu}{x}$ is $m$-stationary in $\pkld{\mu}{x}$.
\end{itemize}
\end{enumerate}
\end{definition}\vspace{1mm}

Here, the stipulation is that  $\kappa$ and $x \cap \kappa$ must be  regular cardinals. In our definition below, we introduce a refinement, by instead requiring $\kappa$ and $|x \cap \kappa|$ are regular limit cardinals. As a result of this refinement, the converse of Proposition \ref{ct} also holds, laying the groundwork for the results on this paper.  \vspace{1mm}

\begin{definition}\label{nst} 
Let $\kappa$ be a regular limit cardinal and let $A$ be a set with $\kappa \subseteq A$.
\begin{enumerate}
\item $S\subseteq \pk{A}$ is \textbf{$0$-stationary} in $\pk{A}$ iff  $S$  is cofinal in $\pk{A}$.
\item  Let $0<n<\omega$. $S\subseteq \pk{A}$ is \textbf{$n$-stationary} in $\pk{A}$ iff for all $m<n$ and  all $T \subseteq \pk{A}$ $m$-stationary in $\pk{A}$, there is $x \in S$ such that 
\begin{itemize}
\item[(i)] $\mu := |x \cap \kappa|$ is a regular limit cardinal,
\item[(ii)] $T\cap \pkld{\mu}{x}$ is $m$-stationary in $\pkld{\mu}{x}$.
\end{itemize}
\item For any $T\subseteq \pk{A} $, $d_n(T):=\{ x \in \pk{A}: \mu:= |x \cap \kappa|$ is a regular limit cardinal and $T \cap \pkld{\mu}{x} \text{ is $n$-stationary in } \pkld{\mu}{x}\}$.
\end{enumerate}\vspace{1mm}
\end{definition}

\begin{lemma} \label{1sticf} 
If $S \subseteq \pk{A}$ is $1$-stationary in $ \pk{A}$, then $S$ is cofinal in  $\pk{A}$. 
\end{lemma} 

\begin{proof} 
Suppose  $S \subseteq \pk{A}$ is  $1$-stationary and let $y \in \pk{A}$. The set $y{\uparrow} \mathrel{=}  \{z \in \pk{A} : y \subseteq z\}$ is club, and in particular cofinal in $\pk{A}$. By $1$-stationarity of $S$, there is  $x \in S$ such that $\mu := |x \cap \kappa| $ is a regular limit cardinal and $y{\uparrow}\cap \pkld{\mu}{x}$ is cofinal in $\pkld{\mu}{x}$. Then $\bigcup (y{\uparrow} \mathrel{\cap}  \pkld{\mu}{x})=x$, and so  $y \subseteq x$. 
\end{proof} \vspace{1mm} 

\begin{definition}\label{nsst} 
Let $\kappa$ be a regular limit cardinal and let $A$ be a set with $\kappa \subseteq A$.
\begin{enumerate}
\item $S\subseteq \pk{A}$ is \textbf{$0$-simultaneously-stationary}  ($0$-s-stationary for short) in $\pk{A}$ iff  $S$  is cofinal in $\pk{A}$.
\item Let $0<n<\omega$. $S\subseteq \pk{A}$ is \textbf{$n$-simultaneously-stationary} ($n$-s-stationary for short) in $\pk{A}$ iff for every $m<n$ and every $T_1,T_2 \subseteq \pk{A}$ $m$-s-stationary in $\pk{A}$, there is some $x \in S$ such that
\begin{itemize}
\item[(i)] $\mu:= |x \cap \kappa|$ is a regular limit cardinal,
\item[(ii)] $T_1 \cap \pkld{\mu}{x}$ and $T_2 \cap \pkld{\mu}{x}$ are both $m$-s-stationary in $\pkld{\mu}{x}$.
\end{itemize}
\item For any $T\subseteq \pk{A} $, $d_n^s(T):=\{ x \in \pk{A} : \mu:= |x \cap \kappa|$ is a regular limit cardinal and $T \cap \pkld{\mu}{x}$ is $n$-s-stationary in $\pkld{\mu}{x}\}$.
\end{enumerate}\vspace{1mm}
\end{definition}

We can also define a $``$simultaneous stationary$"$ version of Definition \ref{nst1}. To do so,  in Definition \ref{nsst} (1) and (2), replace $``$$n$-simultaneously-stationary$"$ with  $``$$n$-simultaneously-stationary (BFS)$"$, and replace $``\mu:= |x \cap \kappa|$ is a regular limit cardinal$"$ with $``\mu:= x \cap \kappa$ is a regular  cardinal$"$.\\

The adjustment made in Definition \ref{nst}, compared to Definition \ref{nst1}, does not introduce a significant difference for the purposes of this paper. This is because the main results in \cite{B2}, which we aim to generalise, focus on simultaneous stationarity. And we will prove (Proposition \ref{limitnolimit2.0}), that for $\kappa$ a regular limit cardinal, the two definitions are actually equivalent.\\

Notice that, in Definitions \ref{nst} and \ref{nsst}, the sets $d_n(T)$ and $d_n^s(T)$ depends not only on $T$ and $n$, but on $\kappa$ and $A$ as well. If we denote $d_n(T)$ by  $d_n(\kappa,A,T)$, for every  $T \subseteq \pkld{\mu}{x} \subseteq \pk{A}$ we  have $$d_n(\mu,x,T)= d_n(\kappa,A,T) \cap \pkld{\mu}{x}.$$ 

Therefore, for the sake of readability we will omit the extra notation $``$$d_n(\kappa,A,T)$$"$ and whenever  we want to specify that we are considering $d_n$ acting only on subsets $T$ of $\pkld{\mu}{x} \subseteq \pk{A}$, we will talk about $d_n(T) \cap \pkld{\mu}{x}$. The same applies to $d_n^s$.

From the definitions above, we immediately see that  $S \subseteq \pk{A}$ is $n$-stationary in $\pk{A}$ if and only if for all $m< n$ and every $T$ that is $m$-stationary in $\pk{A}$, we have $S \cap d_m(T) \neq \varnothing$.  Similarly, $S \subseteq \pk{A}$ is $n$-s-stationary in $\pk{A}$ if and only if for all $m< n$ and for all $T_1,T_2$ that are $m$-s-stationary in $\pk{A}$, we have $S \cap d_m^s(T_1) \cap d_m^s(T_2) \neq \varnothing$. \\

Notice that simultaneous $1$-stationarity is a stronger notion than $1$-stationarity; that is, if $S \subseteq \pk{A}$ is $1$-s-stationary  in $ \pk{A}$, then $S$  is  also $1$-stationary  in $ \pk{A}$. Thus, $d_1^s(T)\subseteq d_1(T)$ for every $T\subseteq \pk{A}$. For the case $n=0$, by definition, we actually have  $d_0(A)=d_0^s(A)$  for all $A\subseteq \pk{A}$. A simple but very useful observation is that  if $S \subseteq \pk{A}$ is $n$-s-stationary  (or $n$-stationary) in $ \pk{A}$ and $S \subseteq S'$, then $S'$  is also $n$-s-stationary (or $n$-stationary)  in $ \pk{A}$. Another straightforward fact to prove  is that for any $X,Y\subseteq \pk{A}$ and any $n < \omega$, we have $d_n(X \cap Y) \subseteq d_n(X) \cap d_n(Y)$ and $d_n^s(X \cap Y) \subseteq d_n^s(X) \cap d_n^s(Y)$.\vspace{1mm}

\begin{lemma}\label{nstimst} 
Let $n<\omega$. If $S \subseteq \pk{A}$ is $n$-s-stationary (respectively $n$-stationary) in $ \pk{A}$, then $S$ is $m$-s-stationary (respectively $m$-stationary)   in $ \pk{A}$ for all $m < n$. Consequently, $d_n^s(X)  \subseteq d_m^s(X)$ (respectively $d_n(X)  \subseteq d_m(X)$) for any $X \subseteq \pk{A}$ and $m\leq n$. 
\end{lemma} 

\begin{proof}  
By induction on $n$. The case $n=1$ follows from Lemma \ref{1sticf}.  Suppose we have the result for all $l< n$, take $S \subseteq \pk{A}$  $n$-s-stationary and fix $m< n$. Let $T_1,T_2 \subseteq \pk{A}$ be $l$-s-stationary for some $l< m$. Since $S$ is $n$-s-stationary in $\pk{A}$ and $l<m\leq n$, there is some $x \in S$ such that $\mu := |x \cap \kappa|$ is a regular limit cardinal and $T_1 \cap \pkld{\mu}{x}$, $T_2 \cap \pkld{\mu}{x}$ are $l$-s-stationary in $\pkld{\mu}{x}$. Therefore, $S$ is $m$-s-stationary   in $ \pk{A}$, and  $d_n^s(X)  \subseteq d_m^s(X)$  for any $X \subseteq \pk{A}$. The respective proof for when $S$ is just $n$-stationary in $\pk{A}$ is completely analogous to the previous proof, replacing $``n$-s-stationarity$"$ by $``n$-stationarity $"$ and taking only one $T \subseteq \pk{A}$ to be $l$-stationary for some $l< m$.
\end{proof} \vspace{2mm} 

Recall that a given ordinal $\kappa > \omega$ is weakly Mahlo if and only if  the set $\{ \mu < \kappa : \mu \ \text{is a regular cardinal} \}$ is stationary in $\kappa$. It is easy to show that such a $\kappa$ is a regular limit cardinal. Moreover, since the set of limit cardinals below  a regular limit cardinal $\kappa$ is a club, we also have the following equivalent definition: an  ordinal $\kappa > \omega$ is weakly Mahlo if and only if the set $\{ \mu < \kappa : \mu \text{ is a regular limit cardinal} \}$ is stationary in $\kappa$. We will conveniently use both definitions.\vspace{1mm}

\begin{prop}\label{siwm}  
If $\pk{A}$ is $1$-stationary  in $\pk{A}$, then $\kappa$ is weakly Mahlo.
\end{prop}

\begin{proof} 
Suppose that  $\pk{A}$ is $1$-stationary in $\pk{A}$. We will prove that $R:=\{ \mu < \kappa : \mu \text{ is a regular cardinal} \}$ is stationary in $\kappa$. Let $C$ be a club subset of $\kappa$. Since $\kappa$ is limit cardinal, we may assume that all elements of $C$ are cardinals. Consider the  set $T_C:= \{ y \in \pk{A} : \exists \alpha \in C  \text{ such that }  $ $  y \cap \kappa \subsetneq \alpha \leq |y|\}$. \\

\begin{claim}
$T_C$  is cofinal in $\pk{A}$.
\end{claim}
\begin{claimproof}
Suppose  $y \in \pk{A}$ and let $\alpha \in C$ be an infinite cardinal such that $y \cap \kappa \subsetneq \alpha$. Consider $ \tilde{\alpha}:=\{ \delta \setminus \{0\} : \delta \in \alpha\}$, clearly, $\tilde{\alpha} \cap \kappa = \{\varnothing \}$. Now, let $z:= y \cup \tilde{\alpha} \supseteq y$. Then $z \cap \kappa =(y  \cup \tilde{\alpha}) \cap \kappa = (y \cap \kappa)  \cup ( \tilde{\alpha} \cap \kappa) = (y \cap \kappa)  \cup \{ \varnothing \}  \subsetneq \alpha$. Moreover, $\alpha = |\alpha | = | \tilde{\alpha}| \leq |y \cup \tilde{\alpha}| =|z| $, whence $z \in T_C$. Hence, for every $y\in \pk{A}$, there is $z \in T_C$ such that $y \subseteq z$. \end{claimproof}\vspace{2mm}

Hence, by  $1$-stationary of $\pk{A}$, there is $x \in \pk{A}$ such that  $\mu := |x \cap \kappa|$ is a regular limit cardinal ($\mu \in R$) and $T_C\cap \pkld{\mu}{x}$ is cofinal in $\pkld{\mu}{x}$. \vspace{2mm} 

\begin{claim}
$C \cap \mu$ is unbounded in $\mu$. 
\end{claim}
\begin{claimproof}
 Let $\gamma < \mu$. Then $\gamma \subsetneq \mu= |x \cap \kappa|  \subseteq x \cap \kappa  \subseteq x$. Since $\mu$ is a regular cardinal, we have  $|\gamma|< \mu$. Therfore $\gamma \in \pkld{\mu}{x}$.  Now, because $T_C$  is cofinal in $\pk{A}$, there exists $y \in T_C \cap \pkld{\mu}{x}$ such that $\gamma \subseteq y$.  As $y \in T_C$, there is some $\alpha \in C$ such that $y \cap \kappa \subsetneq \alpha \leq |y|$.  In summary we have  $\gamma \subseteq y \cap \kappa \subsetneq \alpha \leq |y| < \mu$. This is, for $\gamma < \mu$, $\alpha $ is such that $\alpha \in C \cap \mu$ and $\gamma < \alpha< \mu$. 
\end{claimproof}\vspace{2mm}

Since $C$ is closed, from the previous claim we conclude  that $\mu \in C$. Thus, $\mu \in C \cap R$ and so $R$ is stationary in $\kappa$. 
\end{proof} \vspace{1mm} 

\begin{corollary}\label{s1iwm} (\cite{S1}) 
Let $\kappa$ be a regular limit cardinal. If $\pk{A}$ is $1$-stationary (BFS) in $\pk{A}$, then $\kappa$ is weakly Mahlo.
\end{corollary}

\begin{proof} 
Since we are assuming $\kappa$ to be a limit cardinal, we may proceed in analogy to the proof of Proposition  \ref{siwm}. Notice that, the definition of weakly Mahlo we used in the proof of  \ref{siwm}, used $R:=\{ \mu < \kappa : \mu \text{ is a regular cardinal} \}$. So that, to ensure $\mu \in R$, we only need $\mu$ to be a regular cardinal. And we can derive this from the definition of $1$-stationarity (BSF) of $\pk{A}$.
\end{proof}   \vspace{1mm} 

In the ordinal case, having uncountable cofinality was not only a necessary but also a sufficient condition for an ordinal to be $1$-stationary in itself. Consequently, when seeking an analogy in $\pk{A}$, a question arises: May being $\kappa$ weakly Mahlo be not only a necessary condition but also a sufficient condition for $\pk{A}$ to be $1$-stationary in $\pk{A}$? \vspace{1mm}

\begin{prop}\label{ct} 
If $\kappa$ is weakly Mahlo, then  $\pk{A}$ is $1$-stationary in $\pk{A}$.
\end{prop}

\begin{proof} 
Suppose that $\kappa$ is weakly Mahlo. Then the set $R:=\{ \mu < \kappa : \mu$ is a regular limit cardinal$\}$ is stationary in $\kappa$. Let  $T \subseteq \pk{A}$  be cofinal in $\pk{A}$. Consider a  transfinite sequence of elements of $T$ for ordinals  $\alpha < \kappa$, constructed as follows \vspace{2mm}

\quad $x_0 \in T$, 

\quad $x_{\alpha +1} \in T$ is such that $x_{\alpha + 1} \supsetneq x_{\alpha}\cup  \alpha$,

\quad $x_{\alpha} \in T $ is such that  $x_{\alpha} \supsetneq \bigcup_{\gamma <  \alpha} x_{\gamma}$, for $\alpha < \kappa $ limit. \vspace{2mm}

Any such sequence is well defined, as both the successor and limit step can be performed. More specifically, since $T$ is cofinal and $\kappa$ is a regular limit cardinal, $|x_{\alpha}|$, $ | \alpha|< \kappa$ implies $x_{\alpha} \cup \alpha \in \pk{A}$ and $\bigcup_{\gamma <  \alpha} x_{\gamma}  \in \pk{A}$. Notice also that any such sequence  $\langle x_{\alpha} : \alpha < \kappa\rangle \subseteq T$ is strictly increasing. Now, fix any such $\langle x_{\alpha} : \alpha < \kappa\rangle \subseteq T$, and consider  the set $U:=\{ \alpha < \kappa : \exists \beta < \kappa \text{ s.t. } |x_{\beta}|= \alpha \}$.\vspace{2mm}

\begin{claim}
$U$  is unbounded in $\kappa$.
\end{claim}
\begin{claimproof}
Let $\delta < \kappa$. As $\kappa$ is a regular limit cardinal,  we have $|\delta|^+< \kappa$. Then  $x_{|\delta|^+ +1} \supseteq x_{|\delta|^+} \cup |\delta|^+$. Notice that  $ \delta < |\delta|^+  \leq |x_{|\delta|^++1} |  < \kappa$. Then for $\alpha:=|x_{|\delta|^++1} | < \kappa$, there exists $\beta:= |\delta|^++1< \kappa$ such that  $|x_{\beta}|=\alpha > \delta$. Hence, $\alpha \in U$ and $\delta < \alpha <\kappa.$
\end{claimproof}\vspace{2mm}

Since $R$ is stationary in $\kappa$, there is $\mu \in R$ such that $U \cap \mu$ is unbounded in $\mu$. Let us construct a subsequence $\langle x_{\beta_{\alpha}} : \alpha < \mu \rangle$ of $\langle x_{\alpha} : \alpha < \kappa \rangle$, as follows: Pick a sequence $\langle \delta_{\alpha} : \alpha < \mu \rangle$ of elements from $U \cap \mu$, cofinal in $\mu$. For each $\delta_{\alpha}$ with  $\alpha < \mu$, choose $\beta_{\alpha}< \kappa$ such that $ |x_{\beta_{\alpha}}|= \delta_{\alpha}$. So defined,  $\langle \beta_{\alpha} : \alpha < \mu \rangle$ is strictly increasing. For if $\alpha<\alpha'<\mu$, then $\delta_{\alpha} =|x_{\beta_{\alpha}}|<|x_{\beta_{\alpha'}}|=\delta_{\alpha'}$, because  $\langle \delta_{\alpha} : \alpha < \mu \rangle$ is strictly increasing. And $|x_{\beta_{\alpha}}|<|x_{\beta_{\alpha'}}|$ implies  $\beta_{\alpha}<\beta_{\alpha'}$. \vspace{2mm}

Define $x:= \bigcup_{\alpha <\mu} x_{\beta_{\alpha}} $. Then $x$ is the union of $\mu$ many distinct sets, each of cardinality less than $\mu$. In particular, we have $|x|= \mu$. To conclude the proof, we will show that $x$ satisfies conditions of Definition \ref{nsst} (2).\vspace{2mm}

(i) $\mu = |x \cap \kappa|$: Since  $\langle \beta_{\alpha} : \alpha < \mu \rangle$ is strictly increasing,  $ \mu \leq\bigcup_{\alpha < \mu} \beta_{\alpha}$. Then $\mu \subseteq \bigcup_{\alpha < \mu} \beta_{\delta}\subseteq \bigcup_{\alpha < \mu} x_{\beta_{\delta}}=x$ and so $\mu \subseteq x \cap \kappa$. Hence, $\mu \leq |x \cap \kappa| \leq |x| = \mu$.\\

(ii) $T \cap \pkld{\mu}{x}$ is cofinal in $ \pkld{\mu}{x}$: Let $y \in  \pkld{\mu}{x}$. Then $y \subseteq  \bigcup_{\alpha <\mu} x_{\beta_{\alpha}}$ and $|y| < \mu$. As $|x | = \mu$ is regular,  $y $ is not cofinal in $x$. Then $y \subseteq x_{\beta_{\alpha}}$  for some $\alpha < \mu$. But $x_{\beta_{\alpha}} \subseteq \bigcup_{\alpha <\mu} x_{\beta_{\alpha}}=x$ and $|x_{\beta_{\alpha}} | < \mu $. This is, there is $x_{\beta_{\alpha}} \in T \cap \pkld{\mu}{x}$   such that $y \subseteq x_{\beta_{\alpha}}.$ 
\end{proof}\vspace{2mm} 

Notice that in the previous proof, the sequence $\langle x_{{\alpha}} : \alpha < \mu \rangle$ we chose was actually any sequence satisfying $x_0 \in T$, $x_{\alpha} \cup \alpha \subseteq x_{\alpha+1}$ for $\alpha < \kappa$, and  $x_{\alpha} \supsetneq \bigcup_{\gamma <  \alpha} x_{\gamma}$ for $\alpha < \kappa $ limit. Thus, we actually developed a method for constructing an $x\in \pk{A}$ witnessing the reflection of an unbounded set $T$, based on a given sequence in $T$.  Moreover, if in the proof we take some $E \subseteq R$ stationary in $\kappa$ instead of taking $R$, in the first line of the fourth paragraph we would obtain some $\mu \in E$ such that $U \cap \mu$ is unbounded in $\mu$. Since $\mu$ will still be in $R$, we can then proceed with the rest of the proof just exactly as in \ref{ct}.\vspace{2mm}

\begin{corollary}\label{method} 
Let $\kappa$ be weakly Mahlo, $E\subseteq R$ stationary in $\kappa$ and $T$ be cofinal in $\pk{A}$. Let $\langle x_{\alpha} : \alpha < \kappa\rangle$ be a nonempty sequence of elements of $T$ such that $x_{\alpha} \cup \alpha \subsetneq x_{\alpha+1}$ for $\alpha < \kappa$, and  $ \bigcup_{\gamma <  \alpha} x_{\gamma} \subsetneq x_{\alpha} $ for $\alpha < \kappa $ limit. Then, there  an  strictly increasing subsequence $\langle x_{\beta_{\alpha}} : \alpha < \mu \rangle \subseteq \langle x_{\alpha} : \alpha < \kappa \rangle $ such that  $x:=  \bigcup_{\alpha <\mu} x_{\beta_{\alpha}}$ satisfies, $\mu := |x \cap \kappa|=|x|$ is a regular limit cardinal, $\mu \in E$, and  $T \cap \pkld{\mu}{x}$ is cofinal in $ \pkld{\mu}{x}$. This is, $x \in d_0(T)$, $\mu \in E$ and $|x|=|x\cap \kappa|$. $\square$ 
\end{corollary}

From Proposition \ref{siwm} and Proposition \ref{ct} we get a complete characterisation of $1$-stationarity for $\pk{A}$, and also a complete characterisation  of weakly Mahloness in terms of higher stationary sets on $\pk{A}$.

\begin{theorem}\label{col1}  
$\pk{A}$ is $1$-stationary in $\pk{A}$ if and only if $\kappa$ is weakly Mahlo. $\square$\vspace{1mm}
\end{theorem}

Notice that, in the proof of Proposition \ref{ct} we can in fact  start the sequence $\langle x_{\alpha} : \alpha < \kappa\rangle$ with $x_0\supseteq y$ for any given $y\in \pk{A}$. Obtaining  $y \subseteq x$, for $x$ such that  $\mu:=|x \cap \kappa| $ is a regular limit cardinal and $T \cap \pkld{\mu}{x}$ is cofinal in $\pkld{\mu}{x}$. Therefore, if $\kappa$ is weakly Mahlo and $T \subseteq \pk{A}$ is cofinal in $\pk{A}$, the set  $d_0(T)= $ $\{ x\in \pk{A} : \mu:=|x \cap \kappa| $ is a regular limit cardinal and $T \cap \pkld{\mu}{x}$ is cofinal in $\pkld{\mu}{x}\}$ is  cofinal in $\pk{A}$. \vspace{1mm}

\begin{prop}\label{d0cic} 
Let $\kappa$ be weakly Mahlo. If $C$ be a club in $\pk{A}$, then $d_0^s(C)=d_0(C) \subseteq C$.
\end{prop}

\begin{proof} 
Let $x \in d_0(C) $, then $\mu := |x \cap \kappa|$ is a regular limit cardinal and $C\cap \pkld{\mu}{x}$ is cofinal in $\pkld{\mu}{x}$. Then for each $a \in x$, there is some $y_a \in C\cap \pkld{\mu}{x}$ such that $\{a\} \subseteq y_a$. Similarly, for each pair  $w,w'$ of distinct elements of $C$, there is an element $y \in C\cap \pkld{\mu}{x}$ such that $w,w' \subseteq y$. Thus, for every $z \in [C]^2=\{ z \subseteq C : |z|=2\}$, there is some $y \in C$ such that if $w \in z$ then $w \subseteq y$. For each, $z \in  [C]^2$ we will choose one such $y$ and denote this choice by $y_z$. Now,  by induction on $n<\omega$, construct  the following  sequence of subsets of $\pk{A}$,\vspace{2mm}

\quad $ T_0:= \{ y_a : a \in x\}$,

\quad $T_1:= T_0 \cup \{ y_z  :  z \in [T_0]^2 \}$,

\quad $T_n:= T_{n-1} \cup \{y_z : z \in [T_{n-1}]^2\}$.\vspace{2mm}

Define $T:= \bigcup_{n < \omega} T_n \subseteq C$. To prove that $x \in C$, we will first show that $T$ is a directed set of cardinality less than $\kappa$.\vspace{2mm}

\begin{claim}
$T$ is directed.
\end{claim}
\begin{claimproof}
Let $y_w,y_{w'} \in T$, then there are $n,m < \omega $ such that $y_w \in T_n$ and $y_{w'} \in T_m$. If we set $k=m+n+1$ we get $T_n, T_m \subseteq T_{k}$, and so $y_w,y_{w'} \in T_{k}$. Now, for $z := \{y_w,y_{w'}\} \in  [T_{k}]^2$ there is  $y_z \in  \{y_z : z \in [T_{k}]^2\} \subseteq T_{k+1}$  such that  $y_w, y_{w'} \subseteq y_{z}$. Hence, there is $y_z \in T$ such that $y_w \cup y_{w'} \subseteq y_{z}$. 
\end{claimproof}\vspace{2mm}

\begin{claim}
$|T|<\kappa$.
\end{claim}
\begin{claimproof}
First, let us see that  $|T_n|< \kappa$ for all $n< \omega$. It is clear that  $|T_0|= |x|< \kappa$, moreover, $|T_1|=\max\{ |T_0| , |\{ y_z  : z \in [T_0]^2 \}| \}= \max\{ |x| , |[T_0]^2| \} \leq \max\{ |x| , |T_0|, \omega \} < \kappa$. If $n >2$, by induction, asume  $|T_{n-1}|< \kappa$. Then $$|T_n|=\max\{ |T_{n-1}| , |\{ y_z  : z \in [T_{n-1}]^2 \}| \}= \max\{ |T_{n-1}| , |[T_{n-1}]^2| \} <  \kappa.$$ Finally, since  $|T_n|< \kappa$  for all $n < \omega$, $\kappa$ is regular uncountable and $|T|=\sup \{ |T_n| : n < \omega\} $, we have $|T| < \kappa$. 
\end{claimproof}\vspace{2mm}\\

 Now, as $T$ is a directed subset of $C$ of cardinality less than $\kappa$, and $C$ is a closed subset of $\pk{A}$, we have $\bigcup T \in C$. Moreover, observe that $\bigcup T=x $. For if $a \in \bigcup T$, then $a \in y_z$ for some $y_z \in T_n$ and $n< \omega$. But  $y_z \in  \pkld{\mu}{x}$,  then $a \in y_z \subseteq x$. Conversely, if $a \in x$,  there is $y_a \in T_0 \subseteq T$ such that $\{a\} \subseteq y_a $, this is, $a \in \bigcup T$. Hence, $x =\bigcup T  \in  C$. 
 \end{proof}\vspace{1mm} 

\begin{lemma}\label{lem0} 
For every $X,Y\subseteq \pk{A}$ and every $m\leq n$,
\begin{enumerate}
\item[(1)]  $d_n(X \cap Y)\subseteq d_n(X) \cap d_n(Y)$.  			
\item[(2)]  $d_n(X)\subseteq d_m(X)$. 
\item[(3)]   If $\kappa$ is weakly Mahlo and is $C$ is club in $\pk{A}$, then $d_n(C)\subseteq C$.
\end{enumerate}
\end{lemma}

\begin{proof} 
Proof of (1) follows immediate from the definition and (2) is given by Lemma \ref{nstimst}. So let us prove (3) by induction on $n$. Suppose $\kappa$ is weakly Mahlo and  $C$ is club in $\pk{A}$. The case $n=0$ is  Proposition \ref{d0cic}. For $n>0$, if $d_n(C)\subseteq C$, using item (2) we get  that, $d_{n+1}(C) \subseteq d_n(C)\subseteq C$. 
\end{proof}\vspace{1mm} 

\begin{prop}\label{1stist} 
If $S\subseteq \pk{A}$ is $1$-stationary in $\pk{A}$, then $S$ is stationary in $\pk{A}$. 
\end{prop}
 
\begin{proof}  
Suppose that $S$ is $1$-stationary in $\pk{A}$, and let $C$ be a club subset of $\pk{A}$. In particular $C$ is cofinal in $\pk{A}$, then $ \varnothing \neq S \cap d_0(C)$. Hence, using Proposition \ref{d0cic}, we get $ \varnothing \neq S \cap d_0(C) \subseteq S \cap C $. 
\end{proof}\vspace{1mm}

In the ordinal case the equivalence between $1$-stationarity and stationarity holds. Unfortunately, as we will see in the next proposition, this correspondence does not extend to the case of $\pk{A}$. That is, $S\subseteq \pk{A}$ being stationary in $\pk{A}$ does not imply $S$ is $1$-stationary in $\pk{A}$.\vspace{1mm}

\begin{prop}\label{n1stist} 
If $\kappa$ is a regular limit cardinal and $A$ is a set with $\kappa \subseteq A$, then there is a stationary subset $S$ of $\pk{A}$ that is not $1$-stationary in $\pk{A}$. 
\end{prop}

\begin{proof} 
Consider the set $S:=\{ x \in \pk{A} : cof( |x \cap \kappa|) <  |x \cap \kappa| \}$. We will show that $S$ is stationary in $\pk{A}$, and  $S$ is not $1$-stationary in $\pk{A}$. \vspace{2mm}

To prove $S$ is stationary in $\pk{A}$, take $C$ be club in $\pk{A}$. As $C_{\kappa}$ is club in $\pk{A}$, we may actually assume that $C \subseteq C_{\kappa}$. Let $\langle x_n : n < \omega \rangle$   be an increasing sequence of elements of $C$ such that  $x_0 \cap \kappa > \omega$ and $(x_{n-1} \cap \kappa)^+ \subseteq x_n$ for every $0<n < \omega$. Define $x:=\bigcup_{n<\omega} x_n \in C$, then $x \cap \kappa=\bigcup_{n<\omega} (x_n \cap \kappa) > \omega$. But $\langle x_n \cap \kappa : n < \omega \rangle$   is  a strictly increasing $\omega$-sequence of cardinals greater than $\omega$, then $cof\big(x \cap \kappa)= \omega$. Hence, we have that $cof(|x\cap \kappa|)= \omega< |x \cap \kappa| $, and so $x\in C \cap S$. \vspace{2mm}

Now, we will prove that $S$ is not $1$-stationary in $\pk{A}$. Towards a contradiction, suppose it is.  Let $T$ be cofinal subset of $\pk{A}$, then there is some $x \in S \cap d_0(T)$. But, from $x \in  d_0(T)$, we get that  $|x\cap \kappa|$ is a regular limit cardinal, and so  $x \notin S$. This is a contradiction.
\end{proof} \vspace{1mm} 

 While this might seem disappointing at first glance, it does not prevent us from establishing the remaining analogies between $n$-s-stationary sets of $\kappa$ and $n$-s-stationary sets of $\pk{A}$.\vspace{1mm}

\begin{prop}\label{ci1stpkl} 
Let $\kappa$ be weakly Mahlo. If $C$ is a club in $\pk{A}$, then $C$ is $1$-stationary. 
\end{prop} 

\begin{proof} 
Let $C$ be a club in $\pk{A}$, and suppose that $\kappa$ is weakly Mahlo. Let $T$ be a cofinal subset of $\pk{A}$, we construct a sequence by transfinite induction on ordinals  $\alpha < \kappa$ as follows: \vspace{2mm}

\quad  $x_0 \in T$,  and  $y_0 \in  C$  such that $y_0 \supseteq x_0$, 

\quad $x_{\alpha +1} \in T$ such that $x_{\alpha + 1} \supsetneq x_{\alpha}\cup \alpha \cup y_{\alpha}$, and $y_{\alpha+1} \in  C$  such that $y_{\alpha+1} \supseteq x_{\alpha+1}$,

\quad $x_{\alpha} \in T $ such that  $x_{\alpha} \supsetneq \bigcup_{\gamma <  \alpha} x_{\gamma}$, for $\alpha < \kappa $ limit. \vspace{2mm}

Clearly, this sequence satisfies conditions of Corollary \ref{method}. Therefore, there is a strictly increasing subsequence $\langle x_{\beta_{\alpha}}:\alpha < \mu \rangle $ such that $x:=  \bigcup_{\alpha <\mu} x_{\beta_{\alpha}} \in d_0(T)$. Moreover, notice that, the corresponding sequence $\langle y_{\beta_{\alpha}} : \alpha < \mu\rangle$ of elements of $C$, is also  increasing. And since $C$ is closed, we have $y:= \bigcup_{\alpha <\mu} y_{\beta_{\alpha}} \in C$. \vspace{2mm}

Thus, to conclude the proof it is enough to prove that $x=y$. So let us prove both inclusions. Let $a \in x=\bigcup_{\alpha <\mu} x_{\beta_{\alpha}}$, this is, $a \in  x_{\beta_{\alpha}}$ for some ${\alpha <\mu}$. By construction $x_{\beta_{\alpha}} \subseteq y_{\beta_{\alpha}}$, then $a \in y_{\beta_{\alpha}} \subseteq \bigcup_{\alpha <\mu} y_{\beta_{\alpha}}=y$. Conversely, let $a \in y=\bigcup_{\alpha <\mu} y_{\beta_{\alpha}}$, this is,  $a \in y_{\beta_{\alpha}}$ for some ${\alpha <\mu}$. We have $ x_{\beta_{\alpha}}  \subsetneq  x_{\beta_{\alpha+1}}$ for all $\alpha < \mu$, therefore $ x_{\beta_{\alpha}+1}  \subseteq  x_{\beta_{\alpha+1}}$ for all $\alpha < \mu$. Moreover, by construction, we also have that $y_{\beta_{\alpha}}  \subseteq  x_{\beta_{\alpha}+1} \subseteq  x_{\beta_{\alpha+1}}$. Then $a \in x_{\beta_{\alpha+1}} \subseteq \bigcup_{\alpha <\mu} x_{\beta_{\alpha}}=x$. 
\end{proof}\vspace{1mm} 

\begin{prop}\label{d01} 
Let $\kappa$ be weakly Mahlo. Then,
\begin{enumerate} 
\item If $T_1, \dots, T_q$ are cofinal in $\pk{A}$ for some $q<\omega$, then $d_0(T_1) \cap \cdots \cap d_0(T_q)$ is $1$-stationary in $\pk{A}$.
\item If $T\subseteq \pk{A}$, then $d_0(d_0(T))\subseteq d_0(T)$. 
\end{enumerate}
\end{prop}

\begin{proof} 
(1) Let $T_0$ be cofinal in $\pk{A}$. We will prove that  $\varnothing  \neq d_0(T_0) \cap d_0(T_1) \cap \cdots \cap d_0(T_q)$. Consider the sequences $\mathcal{S}_i=\langle x^{i}_{\alpha} : \alpha < \kappa\rangle $ for $0\leq i\leq q$, constructed by transfinite induction on ordinals  $\alpha < \kappa$, as follows: \vspace{2mm}

\quad $x^{i}_{0} \in T_i$, where  $x^{i}_{0}  \subseteq x^{j}_{0}$ for all $0\leq i \leq j \leq q$, \vspace{2mm}

\quad $x^{i}_{{\alpha +1}} \in T_i$, where $x^{{0}}_{{\alpha + 1}} \supsetneq \alpha \cup \mathop{\bigcup}\limits_{_{\tiny 0 \leq i \leq q}}x^{{i}}_{\alpha}$ and $ x^{i}_{\alpha+1}  \subseteq x^{j}_{\alpha+1}$ for $0<i \leq j \leq q$,

\quad $x^{i}_{\alpha} \in T_i $, where  $x^{i}_{\alpha} \supsetneq \mathop{\bigcup}\limits_{_{\gamma <  \alpha}} x^{i}_{{\gamma}}\ $ for $\alpha $ limit. \vspace{1mm}

$\mathcal{S}_i$ is well defined for each $i \leq q$, because $T_j$ is cofinal for all $j \leq q$ and $\kappa$ is regular. Therefore, each sequence $\mathcal{S}_i$ and  cofinal set $T_i$, satisfies the hypothesis of Corollary \ref{method}. Then for every $i \leq q$, there is a strictly increasing subsequence $\langle x^{i}_{\beta_{\alpha}} : \alpha < \mu \rangle$ such that $x^{i}:=  \bigcup_{\alpha <\mu} x^{i}_{\beta_{\alpha}}  \in d_0(T_i)$.\\

If we prove that $x:=x^{0}= \cdots = x^{i}$, then $x \in d_0(T_0) \cap \cdots \cap d_0(T_q)$ and we are done. So, fix $i< q$, and let us prove that $x^{i}=x^{{i+1}}$. Let $a \in x^{i}=\bigcup_{\alpha <\mu} x^{i}_{\beta_{\alpha}}$, this is, $a \in  x^{i}_{\beta_{\alpha}}$ for some ${\alpha <\mu}$. By construction $x^{i}_{\beta_{\alpha}} \subseteq x^{{i+1}}_{\beta_{\alpha}}$, then $a \in x^{{i+1}}_{\beta_{\alpha}} \subseteq \bigcup_{\alpha <\mu} x^{{i+1}}_{\beta_{\alpha}}=x^{i+1}$. Conversely, if $a \in x^{{i+1}} =\bigcup_{\alpha <\mu} x^{{i+1}}_{\beta_{\alpha}}$, this is, $a \in x^{{i+1}}_{\beta_{\alpha}}$ for some ${\alpha <\mu} $. We have $ x^{i}_{\beta_{\alpha}}  \subsetneq  x^{i}_{\beta_{\alpha+1}} $  for all $\alpha < \mu$, then $ x^{i}_{\beta_{\alpha}+1}  \subseteq  x^{i}_{\beta_{\alpha+1}}$. Moreover, by construction we also have that $x^{^{i+1}}_{\beta_{\alpha}}  \subseteq  x^{i}_{\beta_{\alpha}+1} \subseteq  x^{i}_{\beta_{\alpha+1}}$. Therefore $a \in x^{i}_{\beta_{\alpha+1}} \subseteq \bigcup_{\alpha <\mu} x^{i}_{\beta_{\alpha}}= x^{i}$. 

\quad\\
(2) Let $x \in d_0(d_0(T))$. This is, $\mu:=|x \cap \kappa|$ is a regular limit cardinal and $d_0(T) \cap \pkld{\mu}{x}$  is cofinal in $\pkld{\mu}{x}$. Then, to prove $x \in d_0(T)$, we only have to prove that $T \cap \pkld{\mu}{x}$ is cofinal in $\pkld{\mu}{x}$. Let $y \in \pkld{\mu}{x}$. Since $\mu$ is a limit cardinal, we have that $|y|^+ < \mu $, and so $y \cup |y|^+ \in \pkld{\mu}{x}$. Then, there is some $z \in d_0(T) \cap \pkld{\mu}{x}$ such that $y \cup |y|^+  \subseteq z$. Given that $z \in d_0(T)$, we know that $\gamma:=|z \cap \kappa|$ is a regular limit cardinal and  $T \cap \pkld{\gamma}{z}$ is cofinal in $ \pkld{\gamma}{z}$. Note that $y \subseteq z$ and $|y| < |y|^+ \leq |z \cap \kappa|= \gamma$, then $y \in \pkld{\gamma}{z}$. So that there is some $w \in T \cap \pkld{\gamma}{z}$ such that $y \subseteq w$. Since $T \cap \pkld{\gamma}{z} \subseteq T \cap \pkld{\mu}{x}$, we actually have that $w \in T \cap \pkld{\mu}{x}$ is such that $y \subseteq w$, this is,  $T \cap \pkld{\mu}{x}$ is cofinal in $\pkld{\mu}{x}$.  
\end{proof} 

\begin{prop}\label{d01q1st} 
If $S$ is $1$-stationary  in $\pk{A}$ and $T_1, \dots , T_q\subseteq \pk{A}$ are cofinal in $\pk{A}$ for some $q< \omega$, then $S\cap d_0(T_1)\cap \dots \cap d_0(T_q)$ is $1$-stationary in $\pk{A}$. 
\end{prop}

\begin{proof}  
Let $T_0$ be cofinal in $\pk{A}$, and define $S':= S \cap d_0(T_1) \cap \cdots \cap d_0(T_q)$. We will prove that $d_0(T_0) \cap S' \neq \varnothing$. Since $S\subseteq \pk{A}$ is $1$-stationary in $\pk{A}$, by Proposition \ref{siwm},  we have that $\kappa$ is weakly Mahlo. Then we may apply Proposition \ref{d01} to $T_1, \dots , T_q$, obtaining thereof that $T:=d_0(T_1) \cap \cdots \cap d_0(T_q)$ is $1$-stationary in $\pk{A}$, in particular cofinal. Thus, by  $1$-stationarity of $S$, we have that $\varnothing \neq S \cap d_0(T) $. Hence, by substituting $T$ in this last equation, we obtain 
\begin{align*}
\varnothing &\neq S \cap d_0(\pq{d_0(T_0) \cap d_0(T_1) \cap \cdots \cap d_0(T_q)})  \\  &\subseteq S \cap d_0(d_0(T_0)) \cap d_0(\pq{d_0(T_1)}) \cap \cdots \cap  \hspace{1mm} d_0(\pq{d_0(T_q)}) \\ & \subseteq   S \cap d_0(T_0) \cap d_0(T_1) \cap \cdots \cap d_0(T_q) \\ &=d_0(T) \cap S'
\end{align*} 
where in the first inclusion we used the property  Lemma \ref{lem0} (1), and in the final one we used Proposition \ref{d01} (2). 
\end{proof}\vspace{1mm} 

\begin{corollary}\label{l1} 
Let $\kappa$ be weakly Mahlo. If $S$ is cofinal in $\pk{A}$ and $d_0(S) \subseteq S$, then $S$ is $1$-stationary. $\square$ \vspace{1mm}
\end{corollary}

\begin{prop}\label{1sstiff1st} 
$S$ is $1$-stationary in $\pk{A}$ if and only if $S$ is $1$-s-stationary in $\pk{A}$. In particular, for any $X \subseteq \pk{A}$, $d_1(X)=d_1^s(X)$. 
\end{prop}

\begin{proof} 
The implication from right to the left follows immediately. Suppose now that $S$ is $1$-stationary in $\pk{A}$ and let $T_1,T_2$ be cofinal subsets of $\pk{A}$. By Proposition \ref{d01q1st}, we have  $S \cap d_0(T_1) \cap d_0(T_2)$ is  $1$-stationary, in particular $\varnothing \neq  S \cap d_0(T_1) \cap d_0(T_2)=S \cap d_0^s(T_1) \cap d_0^s(T_2)$, this is,  $S$ is $1$-s-stationary in $\pk{A}$. 
\end{proof}\vspace{1mm} 

Notice that, by Lemma \ref{lem0} (2), for any $X \subseteq \pk{A}$, we have $d_1(d_0(X)) \subseteq d_0(d_0(X))$. Thus, if $\kappa$ is weakly Mahlo, by Proposition \ref{d01} (2), for any $X \subseteq \pk{A}$ we have  $d_1(d_0(X)) \subseteq d_0(X)$. Now, Proposition \ref{1sstiff1st} is telling us that we can actually see the results on $0$- and $1$-stationarity as results on $0$- and $1$-s-stationarity. Then, we have respective versions of our previous results.\vspace{1mm}
\begin{itemize}
\item[(1)] $\pk{A}$ is $1$-s-stationary in $\pk{A}$ if and only if $\kappa$ is weakly Mahlo. 
\item[(2)] If $\kappa$ is weakly Mahlo, then $C \subseteq \pk{A}$ club implies $C$ is $1$-s-stationary. 
\item[(3)] Let $\kappa$ be weakly Mahlo, and let $T_1, \dots, T_l$ be cofinal in $\pk{A}$ for some $l<\omega$. Then $d_0^s(T_1) \cap \cdots \cap d_0^s(T_l)$ is $1$-s-stationary in $\pk{A}$ and $d_0^s(d_0^s(T))\subseteq d_0^s(T)$. 
\item[(4)] If $S\subseteq \pk{A}$ is $1$-s-stationary and $T_1, \dots , T_k\subseteq \pk{A}$ are cofinal in $\pk{A}$, then $S\cap d_0^s(T_1)\cap \dots \cap d_0^s(T_l)$ is $1$-s-stationary in $\pk{A}$.  
\item[(5)] If $\kappa$ is weakly Mahlo, for any $X \subseteq \pk{A}$, $d_1^s(d_0^s(X)) \subseteq d_0^s(X)$. 
\end{itemize}\vspace{2mm}

So far, under the assumption that $\kappa$ is weakly Mahlo, we have established a reasonably reliable correspondence between the base cases $0$ and $1$ of high stationarity on a single ordinal, and  the base cases $0$ and $1$ of high stationarity on $\pk{A}$. Moreover, despite not precisely corresponding to the standard notion of stationarity, we have found (Proposition \ref{strsti1st}) that for $\kappa$  Mahlo, the concept of $1$-stationarity in $\pk{A}$ aligns with a notion of stationarity in $\pk{A}$ explored in \cite{BC2}, known as strongly-stationarity (Definition \ref{strst} below). Consequently the concept of $n$-stationarity,  when considered as an iterative process, does not correspond with the iteration of the  standard  notion of stationarity in $\pk{A}$; instead it might correspond to the iteration of strong stationarity. \\
 
 Recall that a cardinal $\kappa$ is strongly Mahlo (or  Mahlo) if it is a weakly Mahlo and strong limit. When $\kappa$ is a strong limit cardinal and $x \in \pk{A}$, then $\mu:=|x\cap \kappa| \leq |x| < \kappa$ implies $2^{\mu}< \kappa$, and so $|\pkld{\mu}{x}|< \kappa$. \vspace{1mm}

\begin{definition}[See \cite{BC2}]\label{strst} 
Let $\kappa$ be a Mahlo cardinal, we say that $S\subseteq \pk{A}$ is strongly stationary if and only if for all $f : \pk{A} \rightarrow \pk{A}$, there is $x \in S$ such that $\mu:= |x \cap \kappa|$ is a regular limit cardinal and $ f[\pkld{\mu}{x}] \subseteq \pkld{\mu}{x}$.
\end{definition}

\begin{lemma}\label{l2}  
Let $\kappa$ be a Mahlo cardinal, $f : \pk{A} \rightarrow \pk{A}$, and $$T_f:=\{ x \in \pk{A} : \mu:= |x \cap \kappa| \text{ is a regular limit cardinal and } f[\pkld{\mu}{x}] \subseteq \pkld{\mu}{x} \}.$$ Then,  \emph{(1)} $d_0(T_f) \subseteq T_f$ and \emph{(2)} $T_f$ is cofinal in $\pk{A}$. 
\end{lemma}

\begin{proof} 
(1) Let $x \in d_0(T_f)$, then $\mu:= |x \cap \kappa|$ is a regular limit cardinal and $T_f \cap \pkld{\mu}{x}$ is cofinal in $\pkld{\mu}{x}$. It remains to prove that $f[\pkld{\mu}{x}] \subseteq \pkld{\mu}{x}$. Let $y \in \pkld{\mu}{x}$, then, there is $z \in T_f \cap \pkld{\mu}{x}$ such that $y \cup |y|^+ \subseteq z$. Whence, $y \in \pkld{|z \cap \kappa|}{z}$ and  $f(y) \in f[\pkld{|z \cap \kappa|}{z}] \subseteq \pkld{|z \cap \kappa|}{z} \subseteq \pkld{\mu}{x}$. Then $f(y) \in \pkld{\mu}{x}$. \vspace{2mm}

(2) Let $y_0  \in \pk{A}$. Note that $\kappa$ is in particular weakly Mahlo, hence, $d_0(\pk{A})$ is $1$-stationary in $\pk{A}$. Moreover, by Propositions \ref{ci1stpkl} and \ref{d01q1st}, we have that $d_0(\pk{A})\cap C_{\kappa}$  is $1$-stationary in $\pk{A}$. Thus, $ R_{\kappa}:=\{ x \in \pk{A} : x \cap \kappa$ is a regular limit cardinal$\} \supseteq d_0(\pk{A})\cap C_{\kappa}$ is cofinal in $\pk{A}$. Let $x_0 \in  R_{\kappa}$ be such that $y_0 \subseteq x_0$. Consider the following sequence for $\alpha < \kappa$, where $\mu_{\alpha}:= x_{\alpha} \cap \kappa$, \vspace{2mm}

\quad $x_1 \in R_{\kappa}$ is such that $x_1 \supsetneq x_0 \cup \bigcup f[\pkld{\mu_0}{x_0}]$, 

\quad $x_{\alpha+1} \in  R_{\kappa}$ is such that $x_{\alpha+1} \supsetneq x_{\alpha} \cup \bigcup f[\pkld{\mu_{\alpha}}{x_{\alpha}}] \cup \alpha$,

\quad $x_{\alpha} \in  R_{\kappa}$ is such that $x_{\alpha} \supsetneq  \bigcup_{\beta< \alpha} x_{\beta} $, for $\alpha < \kappa$ limit.\vspace{2mm}

Since $\kappa$ is Mahlo, $ \bigcup f[\pkld{\mu_{\alpha}}{x_{\alpha}}]  \in \pk{A}$, and so  $\langle x_{{\alpha}} : \alpha < \kappa \rangle$  is well defined. Moreover, it satisfies the conditions of Corollary \ref{method}. Then, there is a strictly increasing subsequence $\langle x_{\beta_{\alpha}} : \alpha < \mu \rangle \subseteq \langle x_{\alpha} : \alpha < \kappa \rangle $ such that $x:=  \bigcup_{\alpha <\mu} x_{\beta_{\alpha}}  \in \pk{A}$, $\mu := |x \cap \kappa| $ is a regular limit cardinal, $|x|=\mu$ and  $T \cap \pkld{\mu}{x}$ is cofinal in $ \pkld{\mu}{x}$. \vspace{2mm}

We claim that $f[\pkld{\mu}{x}] \subseteq \pkld{\mu}{x}$. Let $z \in \pkld{\mu}{x}$, since $\mu$ is limit cardinal, we have that $z \cup |z|^+ \in \pkld{\mu}{x}$. Hence, since $|x|=\mu$, we have that  $z \cup |z|^+ \subseteq x_{\beta_{\alpha}}$ for some $\alpha < \mu$. Recall that $ x_{\beta_{\alpha}} \in  R_{\kappa}$. Then $|z|<|z|^+ \leq  x_{\beta_{\alpha}} \cap \kappa =  \mu_{\beta_{\alpha}}$, and so $z \in \pkld{ \mu_{\beta_{\alpha}}}{ x_{\beta_{\alpha}}}$. Then we have $f(z) \in f[\pkld{ \mu_{\beta_{\alpha}}}{ x_{\beta_{\alpha}}}]$. Moreover,  $f(z) \subseteq \bigcup f[\pkld{ \mu_{\beta_{\alpha}}}{ x_{\beta_{\alpha}}}] \subseteq x_{\beta_{\alpha}+1}\subseteq x$ and $|f(z)| \leq |x_{\beta_{\alpha}+1}| < \mu$. Therefore $f(z) \in \pkld{\mu}{x}$, and so $f[\pkld{\mu}{x}] \subseteq \pkld{\mu}{x}$.\vspace{2mm}

In summary, we have that  $x\in \pk{A}$ is such that $\mu := |x \cap \kappa| $ is a regular limit cardinal,  $f[\pkld{\mu}{x}] \subseteq \pkld{\mu}{x}$ and $y_0\subseteq x$. Hence $y_0 \subseteq x$, for $x \in T_f$. 
\end{proof}\vspace{1mm} 

\begin{corollary}
Let $\kappa$ be a Mahlo cardinal, and $f : \pk{A} \rightarrow \pk{A}$. Then $T_f$ is $1$-stationary in $\pk{A}$.
\end{corollary}

 \begin{proof} 
 Let $\kappa$ be a Mahlo cardinal and let $U \subseteq \pk{A}$ be cofinal in $\pk{A}$. By Lemma \ref{l2}, we have, $T_f$ cofinal in $\pk{A}$ and $d_0(T_f)\subseteq T_f$. Then by Proposition \ref{d01q1st}, we have that $d_0(U)\cap d_0(T_f)$ is $1$-stationary in $\pk{A}$. Then  $\varnothing \neq d_0(U)\cap d_0(T_f)\subseteq d_0(U) \cap T_f$, whence $T_f$ is $1$-stationary in $\pk{A}$. 
 \end{proof}\vspace{1mm}

\begin{prop}\label{strsti1st}  
If $\kappa$ is Mahlo, then $S\subseteq \pk{A}$ is $1$-s-stationary in $\pk{A}$ if and only if $S\subseteq \pk{A}$ is $1$-stationary in $\pk{A}$, if and only if $S$ is strongly-stationary in $\pk{A}$. 
\end{prop}

\begin{proof} 
Notice that, by proposition \ref{1sstiff1st} the first equivalence always holds. So let us prove that for $\kappa$ Mahlo, $S\subseteq \pk{A}$ is $1$-stationary in $\pk{A}$ if and only if $S$ is strongly-stationary in $\pk{A}$. \vspace{2mm}

$(\Rightarrow)$ Suppose that $S$ is $1$-stationary, and let $f : \pk{A} \rightarrow \pk{A}$. By Lemma \ref{l2}, $T_f$ is cofinal. Then $ \varnothing \neq S \cap d_0(T_f)  \subseteq S \cap T_f$. Therefore, there exists some $x \in  S \cap T_f$. That is, there is $x \in S$ such that $\mu:= |x \cap \kappa|$ is a regular limit cardinal and $ f[\pkld{\mu}{x}] \subseteq \pkld{\mu}{x}$. Hence, $S$ is strongly stationary in $\pk{A}$.\vspace{2mm}

$(\Leftarrow)$ Suppose now that $S$ is strongly stationary and let $T \subseteq \pk{A}$ be cofinal in $\pk{A}$. For each $y \in \pk{A}$, fix  $x_y \in T$ such that $y \subseteq x_y$, and define $f: \pk{A} \rightarrow \pk{A}$ such that $f(y):= x_y$. By strongly stationarity of $S$, there is some $x\in S $ such that $\mu:= |x \cap \kappa|$ is a regular limit cardinal and $ f[\pkld{\mu}{x}] \subseteq \pkld{\mu}{x}$. We claim that $T \cap \pkld{\mu}{x}$ is cofinal in $\pkld{\mu}{x}$. For if $y \in \pkld{\mu}{x}$, then $y \subseteq f(y) = x_y \in T$ and $f(y) \in f[\pkld{\mu}{x}] \subseteq \pkld{\mu}{x}$. This is, $x_y \in T \cap \pkld{\mu}{x}$ is such that $y \subseteq X_y$. 
\end{proof} \vspace{1mm} 

We now turn our attention to exploring higher stationarity in $\pk{A}$ for the cases $n \geq 1$.  \vspace{1mm}

 \begin{prop}\label{new20}  
 Let $n \geq 1$. If $S$ is $n$-stationary and $C$ is club in $\pk{A}$, then $C \cap S$ is  $n$-stationary in $\pk{A}$. 
 \end{prop}
 
\begin{proof}  
Since $n \geq 1$,  $\pk{A}$ is $n$-stationary  implies $\kappa$ is weakly Mahlo. We proceed by induction on $n$. For $n=1$, suppose $S$ is $1$-stationary and  $C$ is club in $\pk{A}$. Notice that $C$ is in particular cofinal. Then by Proposition \ref{d01q1st}, $S\cap d_0(C)$ is $1$-stationary in $\pk{A}$. By Proposition \ref{d0cic} we have that  $S\cap d_0(C) \subseteq S \cap C$. Then  $S \cap C$ is also $1$-stationary in $\pk{A}$. \vspace{2mm} 

Now, assume  the proposition holds for all $m$ such that $1 \leq m < n$. Suppose  $\pk{A}$ is $n$-stationary and  $C$ is club in $\pk{A}$. Let $m < n$ and  let $T$ be $m$-stationary in $\pk{A}$. By induction hypothesis, we have that $C \cap T$ is  $m$-stationary in $\pk{A}$. Then $$\varnothing \neq S \cap d_{m}(C \cap T)\subseteq  S \cap d_{m}(C) \cap d_{m}(T) \subseteq  S \cap C \cap d_{m}(T)$$ 
where in the last  inclusions we used Lemma  \ref{lem0} (1) and (3). Hence, we have $ \varnothing \neq S \cap C \cap d_{m}(T)$, and so $C \cap S$ is $n$-stationary in $\pk{A}$. 
\end{proof} \vspace{1mm} 

\begin{corollary}
Let $n < \omega$.  If $\pk{A}$ is $n$-stationary  and $C$ is club in $\pk{A}$, then $C$ is also $n$-stationary. $\square$  \vspace{1mm}
\end{corollary}

\begin{lemma}\label{lem1}  
For every $X,Y\subseteq \pk{A}$ and every $m\leq n$,
 \begin{enumerate}
\item[(1)]  $d_n^s(X \cap Y)\subseteq d_n^s(X) \cap d_n^s(Y)$.  			
\item[(2)]  $d_n^s(X)\subseteq d_m^s(X)$. 
\item[(3)]  $d_n^s(d_m^s(X))\subseteq d_m^s(X)$.
\item[(4)]  If $\kappa$ is weakly Mahlo and $C\subseteq \pk{A}$ club, then $d_n^s(C)\subseteq C$.
\end{enumerate}
\end{lemma}

\begin{proof} 
(1) follows immediate from the definition. (2) is Lemma \ref{nstimst} (2). We prove (4) by induction on $n$. Suppose $\kappa$ is weakly Mahlo and  $C$ is club in $\pk{A}$. Then the case $n=0$ is  Proposition \ref{d0cic}. And if we assume $d^s_n(C)\subseteq C$, then applying  item (2) to $C$ we obtain that $d_{n+1}^s(C) \subseteq d_n^s(C)\subseteq C$. \vspace{2mm}

Now, let us prove (3) by induction on $n$. The case $n=0$ follows from Proposition \ref{d01} (2). The case $n=1$ is immediate, using the remark below Proposition  \ref{1sstiff1st}. So, suppose $n\geq 2$. If $m=0$,  by (2) we get,  $$d_n^s (d_0^s(X))\subseteq d_0^s(d_0^s(X)) \subseteq d_0^s(X)$$ and we are done. Thus, consider $m \geq 1$. Let $x \in d_n^s(d_m^s(X))$, this is $\mu:= |x \cap \kappa|$ is a regular limit cardinal and $d_m^s(X) \cap \pkld{\mu}{x}$ is $n$-s-stationary in $\pkld{\mu}{x}$. We want to prove that $x \in d_m^s(X)$. For this, take $l<m$ and $T_1,T_2$  $l$-s-stationary in  $\pkld{\mu}{x}$. By $n$-s-stationarity of  $d_m^s(X) \cap \pkld{\mu}{x}$, there is some $y \in d_m^s(X) \cap \pkld{\mu}{x}$ such that $\gamma:=|y\cap \kappa|$ is a regular limit cardinal and $T_1 \cap \pkld{\gamma}{y}$ and $T_2 \cap \pkld{\gamma}{y}$   are $l$-stationary in $\pkld{\gamma}{y}$. From $y \in d_m^s(X)$ we have that $X \cap \pkld{\gamma}{y}$ is $m$-s-stationary in $\pkld{\gamma}{y}$. Then, since $l<m$ we have:  
\begin{eqnarray*}
\varnothing &\neq& X \cap \pkld{\gamma}{y} \cap d_l ^s {\big(}\pq{T_1 \cap \pkld{\gamma}{y}}{\big)} \cap d_l^s {\big(}\pq{T_2 \cap \pkld{\gamma}{y}}{\big)} \\ & \subseteq  &X \cap \pkld{\gamma}{y} \cap d_l^s(T_1) \cap d_l^s \big( \pq{\pkld{\gamma}{y}} \big) \cap d_l^s(T_2) \cap d_l^s \big( \pq{\pkld{\gamma}{y}}\big) \\ &\subseteq & X \cap d_l^s(T_1) \cap d_l^s(T_2)
\end{eqnarray*}
where in the first   inclusion we used part (1) of the lemma. Then $\varnothing \neq X \cap d_l^s(T_1) \cap d_l^s(T_2)$, and so $X$ is $m$-s-stationary  in  $\pkld{\mu}{x}$. Hence, $x\in d_0(X)$. 
\end{proof} \vspace{1mm}

 \begin{prop}\label{new2} 
 Let $n \geq 1$. If $S$ is $n$-s-stationary and $C$ is club in $\pk{A}$, then $C \cap S$ is also $n$-s-stationary in $\pk{A}$.  
 \end{prop}
 
\begin{proof} 
Since $n \geq 1$,  $\pk{A}$ is $n$-s-stationary  implies $\kappa$ is weakly Mahlo. We proceed by induction on $n$. The case $n=1$ follows from Proposition \ref{new20}. So, assume  the proposition holds for all $m$ such that $n >m\geq 1$. Let $m \leq n$ and  $T_1,T_2$ $m$-s-stationary in $\pk{A}$. By induction hypothesis, $C \cap T_1$ and $C \cap T_2$ are $m$-s-stationary in $\pk{A}$. Then by $n$-stationarity of $S$, we get  $$\varnothing \neq S \cap d_{m}^s(C \cap T_1)\cap d_{m}^s(C \cap T_2) \subseteq S \cap C \cap d_{m}^s(T_1)\cap d_{m}^s(T_2).$$ In the last  inclusion we used Lemma  \ref{lem1} (1) and (4). Then $ \varnothing \neq S \cap C \cap d_{m}^s(T_1) \cap d_{m}^s(T_2)$, and so $C \cap S$ is $n$-s-stationary in $\pk{A}$. 
\end{proof}\vspace{1mm}

\begin{corollary}
Let $1 \leq n < \omega$.  If $\pk{A}$ is $n$-s-stationary  and $C$ is club in $\pk{A}$, then $C$ is also $n$-s-stationary. $\square$\vspace{1mm}
\end{corollary}

\begin{prop}\label{eqc} 
Let $n < \omega$, $\kappa$ a regular limit cardinal and $\kappa \subseteq A$. Then,  
\begin{enumerate}
\item[(1)] If $S\subseteq \pk{A}$ is $n$-s-stationary and $X_i\subseteq \pk{A}$ is $m_i$-s-stationary for $m_i< n$ and $i \in \{0,...,q-1\}$, then $S \cap d^s_{m_0}(X_{0}) \cap \cdots \cap d^s_{m_{q-1}}(X_{q-1})$ is $n$-s-stationary in $\pk{A}$. 
\item[(2)] If $S\subseteq \pk{A}$ is $(n+1)$-s-stationary in $\pk{A}$, then, for all $m \leq n$ and all $T_1,T_2\subseteq \pk{A}$ $m$-s-stationary in $\pk{A}$, $S \cap d_m^s(T_1) \cap d_m^s(T_2)$ is $m$-s-stationary in $\pk{A}$. 
\end{enumerate}
\end{prop}
 
\begin{proof} 
We will prove (1) and (2)  by simultaneously  induction on $n<\omega$:
\begin{enumerate}
\item[(n=0)] (1) Nothing to prove, because there is no $m_i<0$.\vspace{2mm}\\ 
(2) Suppose $S\subseteq \pk{A}$ is $1$-s-stationary in $\pk{A}$, and let $T_1,T_2\subseteq \pk{A}$ be 	cofinal in $\pk{A}$. Then $\kappa$ is weakly Mahlo, and by Proposition \ref{d01q1st}  $S \cap d^s_0(T_1) \cap d^s_0(T_2)$ is  $1$-s-stationary in $\pk{A}$. In particular, $S \cap d^s_0(T_1) \cap d^s_0(T_2)$ is cofinal in $\pk{A}$. \vspace{2mm}
	
\item[(n=1)] (1) Suppose $S\subseteq \pk{A}$ is $1$-s-stationary in $\pk{A}$, and $X_i\subseteq \pk{A}$ is cofinal in $\pk{A}$ for $i \in \{0,...,q-1\}$. Then  $\kappa$ is weakly Mahlo, and by Proposition \ref{d01q1st}  $S \cap d^s_{0}(X_{0}) \cap \cdots \cap d^s_{0}(X_{q-1})$ is $1$-s-stationary in $\pk{A}$.\vspace{2mm}\\ 
(2) Since $n=1$, we  consider the two possible cases for $m \leq n$ separately: \vspace{2mm}

(m=0): Suppose $S\subseteq \pk{A}$ is $2$-s-stationary in $\pk{A}$, and let $T_1,T_2 \subseteq 	\pk{A}$ be cofinal in $\pk{A}$. In particular, $S$ is $1$-s-stationary. Thus, by Proposition \ref{d01q1st}, we have  $S \cap d^s_0(T_1) \cap d^s_0(T_2)$ is cofinal in $\pk{A}$.\vspace{2mm}

\ (m=1): Suppose $S\subseteq \pk{A}$ is $2$-s-stationary in $\pk{A}$, and let $T_1,T_2\subseteq \pk{A}$ be $1$-s-stationary in $\pk{A}$. Let $U$ be cofinal in $\pk{A}$, we will prove that $d^s_0(U) \cap S \cap d^s_1(T_1) \cap d^s_1(T_2) \neq \varnothing$. Since $\pk{A}$ is $2$-s-stationary, $\kappa$ is weakly Mahlo. Then by Proposition \ref{d01q1st}, we have that  $d^s_0(U) \cap 	T_1$ and $d^s_0(U) \cap T_2$  are both $1$-stationary in $\pk{A}$. Now, by {2-s}-stationarity of $S$, 
\begin{minipage}{\linewidth}   
\begin{eqnarray*} 
\varnothing& \neq &S \cap {d^s_1\big(} \pq{d^s_0(U) \cap T_1}{\big)} \cap {d^s_1\big(} \pq{d^s_0(U) \cap T_2}{\big)} \\ & \subseteq&  S \cap d^s_1\big(d^s_0(U)\big)  \cap d^s_1(T_1) \cap d^s_1\big( d^s_0(U)\big) \cap d^s_1(T_2) \\ & \subseteq & S \cap d^s_0(U) \cap d^s_1(T_1) \cap d^s_0(U) \cap d^s_1(T_2).
\end{eqnarray*} 
\end{minipage}\vspace{2mm}
where in the first inclusion we used Lemma \ref{lem1} (1), and in the last one Lemma \ref{lem1} (3). Then $\varnothing \neq d^s_0(U)  \cap S \cap d^s_1(T_1)  \cap d^s_1(T_2)$, and so $S \cap d^s_1(T_1)  \cap d^s_1(T_2)$ is $1$-stationary. By proposition \ref{1sstiff1st}, this is, $S \cap d^s_1(T_1)  \cap d^s_1(T_2)$ is $1$-s-stationary in $\pk{A}$.
\end{enumerate}\vspace{1mm}
Suppose now that  $(1)$ and $(2)$ hold for $n$, and let us prove they hold for $n+1$:
 \begin{itemize}
\item[(n+1)]   (1)  Suppose that $S\subseteq \pk{A}$ is $(n+1)$-s-stationary in $\pk{A}$ and that for all  $i \in \{0,...,q-1\}$, $X_i\subseteq \pk{A}$ is $m_i$-s-stationary for $m_i< n+1$. We will proceed by induction on $q$. \vspace{2mm}

\begin{enumerate}	
\item[(q=1)] Let $l<n+1$ and let $T_1,T_2$ be $l$-s-stationary  in $\pk{A}$. We need to prove that $S \cap d^s_{m_0} (X_{0}) \cap d^s_{l}(T_1) \cap d^s_{l}(T_2)  \neq \varnothing$. Notice that $m_0,l \leq n$, then we will consider three cases; \vspace{2mm}
		 
($m_0>l$): $X_0$ is $m_0$-s-stationary  in $\pk{A}$ and $m_0\leq n$. Then by induction hypothesis on (2) we have, $S\cap d^s_{m_0}(X_0)$ is $m_0$-s-stationary in $\pk{A}$. Now, since $l < m_0$ and $T_1,T_2$ are $l$-s-stationary in $\pk{A}$, we have

\begin{minipage}{\linewidth}    
\begin{eqnarray*} 
S \cap d^s_{m_0} (X_{0}) \cap d^s_{l}(T_1) \cap d^s_{l}(T_2)  \neq \varnothing.
\end{eqnarray*}
\end{minipage}\vspace{2mm}

($m_0=l$): Note that $l \leq n$, then by induction hypothesis on (2) we have, $S\cap d^s_{l}(T_1) \cap d^s_{l}(T_2) $ is $l$-s-stationary in $\pk{A}$. In particular  $d^s_{l}(T_1) \cap d^s_{l}(T_2) $ is $l$-s-stationary in $\pk{A}$. 

Since $m_0=l$, $X_0$ is also $l$-s-stationary in $\pk{A}$. Then by the $(n+1)$-{s}-stationarity of $S$ we get, 

\begin{minipage}{\linewidth}    
\begin{eqnarray*} 
\varnothing &\neq & S \cap d^s_{l} {(}X_{0}{)} \cap d^s_l {\big(}\pq{d^s_{l}(T_1) \cap d^s_{l}(T_2)}{\big)}\\ & \subseteq &S \cap d^s_{l} (X_{0}) \cap d^s_{l}(T_1) \cap d^s_{l}(T_2). 
\end{eqnarray*}
\end{minipage}\vspace{1mm}

In the last inclusion we used Lemma \ref{lem1} (1) and (3). Since $l=m_0$, we actually have

\( \hspace{30mm} \varnothing \neq S 	\cap d^s_{m_0} (X_{0}) \cap d^s_{l}(T_1) \cap d^s_{l}(T_2). \) \vspace{1mm}

($m_0<l$): On the one hand, $S$ is in particular $n$-s-stationary in $\pk{A}$, $X_0$ is $m_0$-s-stationary  in $\pk{A}$ and $m_0<n$. Then by induction hypothesis on (1) we have, $S\cap d^s_{m_0}(X_0)$ is $n$-s-stationary in $\pk{A}$. In particular $d^s_{m_0}(X_0)$ is $n$-s-stationary, and so $l$-s-stationary in $\pk{A}$ (Lemma \ref{nstimst}). On the other hand, $T_1,T_2$ are $l$-s-stationary  in $\pk{A}$ and $l \leq n$. Thus, by induction hypothesis on (2),  $d^s_{l}(T_1) \cap d^s_{l}(T_2) $ is $l$-s-stationary in $\pk{A}$.  Using the  $(n+1)$-{s}-stationarity of $S$ and  Lemma \ref{lem1} (1) and (3),  we get        \\     \begin{minipage}{\linewidth}    
\begin{eqnarray*} 
\varnothing & \neq & S \cap d^s_{l} {\big(} d^s_{m_0} (X_{0}) {\big)} \cap d^s_l {\big(}\pq{d^s_{l}(T_1) \cap d^s_{l}(T_2)}{\big)} \\ & \subseteq & S \cap d^s_{m_0} (X_{0}) \cap d^s_{l}(T_1) \cap d^s_{l}(T_2).
\end{eqnarray*}
\end{minipage}
\end{enumerate}\vspace{2mm}
Suppose now that (1) holds for $q-1$, and let us prove it for $q$: 
 \begin{enumerate}
\item[(\ q\ )] $S\subseteq \pk{A}$ is $(n+1)$-s-stationary in $\pk{A}$ and $X_i\subseteq 	\pk{A}$ is $m_i$-s-stationary for $m_i< n+1$ and $i \in \{0,...,q-1\}$. By induction hypothesis, $S^*:=S \cap d^s_{m_0} (X_{0}) \cap \cdots \cap d^s_{m_{q-2}} (X_{q-2}) $ is $(n+1)$-s-stationary in $\pk{A}$. Then using by the base case $(q=1)$ of the induction with $S^*$, we get that $S^* \cap d^s_{m_{q-1}} (X_{q-1})$ is $(n+1)$-s-stationary in $\pk{A}$. Whence, $S^* \cap d^s_{m_{q-1}} (X_{q-1})=S \cap d^s_{m_0}(X_{0}) \cap \cdots \cap d^s_{m_{q-1}}(X_{q-1})$ is $(n+1)$-s-stationary in $\pk{A}$. 
\end{enumerate}
\quad\\
(2)  Suppose $S\subseteq \pk{A}$ is $(n+2)$-s-stationary in $\pk{A}$ and let $m \leq n+1$ and $T_1,T_2\subseteq \pk{A}$ be $m$-s-stationary in $\pk{A}$. Consider two cases: \vspace{2mm}
\quad\\
$(m<n+1)$: $S\subseteq \pk{A}$ is in particular $(n+1)$-s-stationary in $\pk{A}$, $m \leq n$ and $T_1,T_2\subseteq \pk{A}$ are $m$-s-stationary in $\pk{A}$. By induction hypothesis on (2), we have that  $S\cap d^s_m(T_1) \cap d^s_m(T_2)$ is $m$-s-stationary.\vspace{2mm}
\quad\\
$(m=n+1)$: Let $l<m=n+1$ and  let $U_1,U_2$ be $l$-s-stationary in $\pk{A}$. $T_1,T_2\subseteq \pk{A}$ are $m$-s-stationary in $\pk{A}$. Then for $i=1,2$, $T_i$ is $(n+1)$-s-stationary in $\pk{A}$ and $U_i$ is  $l$-s-stationary in $\pk{A}$. Let $i=1,2$, then $T_i$, $l,n+1$ and $U_i$ satisfy the conditions required in the hypothesis of (1). Since we already proved (1), it follows that  $T_1\cap d^s_l(U_1)$ and  $T_2\cap d^s_l(U_2)$ are both $(n+1)$-s-stationary in $\pk{A}$. Thus, by $(n+2)$-{s}-stationarity of $S$, \vspace{2mm}

 \begin{minipage}{\linewidth}    
\begin{eqnarray*} 
\varnothing & \neq & S\cap d^s_{n+1}{\big(}\pq{T_1 \cap d^s_l(U_1)}{\big)} \cap d^s_{n+1}{\big(}\pq{T_2 \cap d^s_l(U_2)}{\big)} \\ & \subseteq & S\cap d^s_{n+1}(T_1) \cap d^s_{n+1}(T_2) \cap d^s_l(U_1) \cap d^s_l(U_2).
\end{eqnarray*}
\end{minipage}  \vspace{2mm}
\quad\\
where in the inclusion we used Lemma \ref{lem1}. Since $l$ was arbitrary and such that $l < n+1=m$,  last equation is tell us that $S\cap d^s_m(T_1) \cap d^s_m(T_2)$ is $m$-s-stationary in $\pk{A}$. \qedhere
\end{itemize}
\end{proof} \vspace{1mm} 

Proposition \ref{eqc} mirrors  two of the results presented by Bagaria in \cite{B2} (Proposition 2.10). Its successful generalisation to the case of $\pk{A}$ is a key factor that enabled us to extend Bagaria's results to the case of $\pk{A}$. \\

Let us now establish the equivalence between the $``$simultaneously stationary$"$ versions of Definitions \ref{nst1} and \ref{nst}, whenever $\kappa$ is a regular limit cardinal. Recall that, when $\kappa$ is limit, $C_{\kappa}= \{ x \in \pk{A} : x \cap \kappa $ is  cardinal$\}$ is a club in $\pk{A}$. Therefore, whenever $S$ is $n$-s-stationary, by Proposition  \ref{new2}, we have that $S \cap C_{\kappa}$ is also $n$-s-stationary.\vspace{1mm}

\begin{prop}\label{limitnolimit2.0} 
For all $n< \omega$ and $\kappa$ regular limit cardinal, $S$ is $n$-s-stationary in $\pk{A}$ if and only if  $S$ is $n$-s-stationary (BFS) in $\pk{A}$.
 \end{prop}

\begin{proof} 
We proceed by induction on $n< \omega$. The case $n=0$ holds true because in $0$ both definitions coincide, given that $\kappa$ is a limit cardinal. So, let $n \geq 1$ and  assume the theorem holds for all $m<n$.  \vspace{2mm} 

($\Rightarrow$)  Suppose $S$ is $n$-s-stationary in $\pk{A}$, then $S':=S \cap C_{\kappa}$  is $n$-s-stationary in $\pk{A}$. Let $T_1,T_2$  be $m$-s-stationary (BFS) in $\pk{A}$ for some $m<n$. By induction hypotheses, $T_1,T_2$  are $m$-s-stationary in $\pk{A}$. Then there is $x\in S'$  such that $\mu:= |x \cap \kappa|$ is a  regular limit  cardinal and $T_1 \cap \pkld{\mu}{x}$, $T_2 \cap \pkld{\mu}{x}$ are $m$-s-stationary in $\pkld{\mu}{x}$. Note that $x \in S'\subseteq C_{\kappa}$, then $\mu=x \cap \kappa=|x \cap \kappa|$. Since $\mu$ is a limit cardnal, by induction hypotheses, $T_1 \cap \pkld{\mu}{x}$ and $T_2 \cap \pkld{\mu}{x}$ are also $m$-s-stationary  (BFS) in $\pkld{\mu}{x}$. Therefore,  $x\in S' \subseteq S$  is such that $\mu= x \cap \kappa$ is a regular cardinal, and $T_1 \cap \pkld{\mu}{x}$, $T_2 \cap \pkld{\mu}{x}$ are $m$-s-stationary (BFS)  in $\pkld{\mu}{x}$. This is, $S$ is $n$-s-stationary  (BFS)  in $\pk{A}$. \vspace{2mm}

($\Leftarrow$) Suppose $S$ is $n$-s-stationary (BFS) in $\pk{A}$. Let $T_1,T_2$  be $m$-s-stationary  in $\pk{A}$ for some $m<n$. Let us consider two cases.

\begin{itemize}
\item[(i)] $m=0$: Since $n \geq 1$ and  $S$ is $n$-s-stationary (BFS) in $\pk{A}$, in particular, $S$ is $1$-s-stationary (BFS) in $\pk{A}$. And since $\kappa$ is a limit cardinal, then  by Corollary \ref{s1iwm}, $\kappa$ is weakly Mahlo. Now, by Proposition \ref{d01} and Proposition \ref{new2}, $T':=d_{0}(T_1)\cap d_{0}(T_2)\cap C_{\kappa}$ is $1$-s-stationary in $\pk{A}$. In particular $T'$ is cofinal, and so there exists $x \in S$ such that $\mu:= x \cap \kappa$ is a regular  cardinal, and $T' \cap \pkld{\mu}{x}$ is cofinal in $\pkld{\mu}{x}$. Notice that, since $\mu= x \cap \kappa$ is a cardinal, we  have that $\mu=x \cap \kappa=|x \cap \kappa|$. \vspace{2mm}

We claim that $\mu$ is a  limit cardinal.  Suppose it is not, so $\mu= \delta^+$ for some cardinal $\delta$. Then $ |\delta +1 | < \delta^+=\mu=x\cap \kappa$, and so $\delta+1 \in \pkld{\mu}{x}$. Since $T' \cap \pkld{\mu}{x}$ is cofinal, there exists $y \in T' \cap \pkld{\mu}{x}$ such that $\delta+1 \subseteq y$. Now, because $y \in T' \subseteq C_{\kappa}$, we know that  $|y \cap \kappa|=y\cap \kappa$ is a regular limit cardinal. Therefore, as ordinals, $\delta <\delta+1 \leq y \cap \kappa < \mu = \delta^+$. But then, as cardinals, $\delta < y \cap \kappa <  \delta^+$, which is clearly a contradiction.\vspace{2mm}

Now, since $T' \cap \pkld{\mu}{x}$ is cofinal in $\pkld{\mu}{x}$, then $x \in d_{0}(T')$. By Proposition \ref{d01} (2), we have that $d_{0}\big(\pq{d_{0}(T_1)\cap d_{0}(T_2)\cap C_{\kappa}}\big)\subseteq d_{0}(T_1)\cap d_{0}(T_2) \cap C_{\kappa}$. Then $x \in d_{0}(T_1)\cap d_{0}(T_2)$, and so $T_1 \cap  \pkld{\mu}{x}$ and $T_2 \cap  \pkld{\mu}{x}$  are cofinal in $ \pkld{\mu}{x}$.\vspace{2mm} 

\item[(ii)] $m \geq 1$: By Proposition  \ref{new2}, we have that $T_1':=T_1 \cap C_{\kappa} $ and $T_2':=T_2 \cap  C_{\kappa} $ are also $m$-s-stationary in $\pk{A}$. Then by induction hypothesis,  $T_1'$ and $T_2'$ are also $m$-s-stationary (BFS) in $\pk{A}$. Then, there is $x \in S$ such that $\mu:= x \cap \kappa$ is a regular  cardinal, and $T_1' \cap \pkld{\mu}{x}$ and $T_2' \cap \pkld{\mu}{x}$ are both $m$-s-stationary (BFS) in $\pkld{\mu}{x}$. Notice that, since $\mu= x \cap \kappa$ is a cardinal, we  have that $\mu=x \cap \kappa=|x \cap \kappa|$. Now, since $T_1' \cap \pkld{\mu}{x}$ is in particular cofinal, and $T_1'\subseteq C_{\kappa}$, we prove that $\mu$ is a regular limit cardinal, using an argument analogous to the case $m=0$.\vspace{2mm} 

In summary, we have that  $x$ is such that $\mu= |x \cap \kappa|$ is a regular limit cardinal, and  $T_1' \cap \pkld{\mu}{x}$ and $T_2' \cap \pkld{\mu}{x}$ are  $m$-s-stationary (BFS) in $\pkld{\mu}{x}$. Then by inductive hypothesis,  $T_1' \cap \pkld{\mu}{x}$ and $T_2' \cap \pkld{\mu}{x}$ are  $m$-s-stationary in $\pkld{\mu}{x}$. And because $T_1' \cap \pkld{\mu}{x} \subseteq T_1 \cap \pkld{\mu}{x}$ and $T_2' \cap \pkld{\mu}{x} \subseteq T_2 \cap \pkld{\mu}{x}$, we have that $T_1 \cap  \pkld{\mu}{x}$, $T_2 \cap  \pkld{\mu}{x}$  are $m$-s-stationary in $ \pkld{\mu}{x}$. \qedhere
\end{itemize} 
\end{proof}\vspace{2mm}  

With this equivalence established, we will now proceed to prove that two of the main results on higher s-stationarity in the ordinal case (Theorems 3.1 and 2.11 from \cite{B2}) can be successfully generalised to higher s-stationarity in the $\pk{A}$ case.

\begin{prop}\label{eqddc} 
For every $n< \omega$, the set $\mathrm{NsS}_{\kappa}^n(A):=\{N \subseteq \pk{A} : N $ is not $n$-s-stationary  $\text{ in } \pk{A} \}$ is an ideal on $\pk{A}$. 
\end{prop}

\begin{proof} 
Clearly  $\varnothing \in \mathrm{NsS}_{\kappa}^n(A) $, for all $n<\omega$. Also, it is easy to see that if $N_1 \subseteq N_2$ and $N_2 \in \mathrm{NsS}_{\kappa}^n(A)$ then $N_1 \in \mathrm{NsS}_{\kappa}^n(A)$, for all $n<\omega$. We are left to prove that if $n<\omega$ and $N_1,N_2 \in \mathrm{NsS}_{\kappa}^n(A)$, then $N_1 \cup N_2 \in \mathrm{NsS}_{\kappa}^n(A)$. We consider two cases: \vspace{1mm} 

\begin{enumerate}
\item[(n=0)]  Suppose  $N_1,N_2  \in \mathrm{NsS}_{\kappa}^0(A)$, this is, $N_1,N_2$ are not cofinal in $\pk{A}$. Then, 
\begin{center}
(i) there is $x_1 \in \pk{A}$ such that for all $y \in N_1$, $x_1 \not\subseteq y$, \\
(ii) there is $x_2 \in \pk{A}$ such that for all $y \in N_2$, $x_2 \not\subseteq y$. 
\end{center}
Towards a contradiction, assume that $N_1 \cup N_2$ is cofinal in $\pk{A}$. Then, there is $y_1 \in N_1\cup N_2$ such that $x_1 \subseteq y_1$. Similarly, there is $y_2 \in N_1\cup N_2$ such that $y_1 \cup x_2 \subseteq y_2$. So that, we have   $x_1 \subseteq y_1 \cup x_2 \subseteq y_2$. Now, if $y_2 \in N_1$, then $x_1 \subseteq y_2$, contradicting (i). And, if $y_2 \in N_2$, then $x_2 \subseteq y_2$, contradicting (ii).  Hence, $N_1\cup N_2 $  is not cofinal in $\pk{A}$. \vspace{2mm}
\item[(n>0)] Suppose  $N_1,N_2 \subseteq \pk{A}$ are not $n$-s-stationary in $\pk{A}$. Then, for $i=1,2$, there are $m_i<n$ and $T_i,U_i \ $   $m_i$-s-stationary subsets of $\pk{A}$ such that
$$\hspace{15mm} d^s_{m_i}(T_1) \cap d^s_{m_i}(U_1) \cap N_i = \varnothing. \quad (\dagger)_i$$
Now, if $\pk{A}$ is not $n$-s-stationary, no subset of $\pk{A}$  is $n$-s-stationary. In particular,  $N_1\cup N_2 $  is not $n$-s-stationary in $\pk{A}$. So, suppose  $\pk{A}$ is $n$-s-stationary. Then by Proposition \ref{eqc} (1), we have that $$\hspace{15mm} S:=d^s_{m_1}(T_1) \cap d^s_{m_1}(U_1) \cap d^s_{m_2}(T_2) \cap d^s_{m_2}(U_2)$$ is $n$-s-stationary in $\pk{A}$. Towards a contradiction, suppose that $N_1 \cup N_2$ is $n$-s-stationary. Then, there is  $x \in d^s_{n-1}(S)  \cap {(} N_1 \cup N_2 {)}$. Since $m_1,m_2 \leq n-1$, by substituting $S$ and using Lemma \ref{lem1} (1) and (2), we get 
\begin{minipage}{\linewidth}    
\begin{eqnarray*}
x &\in &d^s_{n-1}{\big(}\pq{d^s_{m_1}(T_1) \cap d^s_{m_1}(U_1) \cap d^s_{m_2}(T_2) \cap d^s_{m_2}(U_2)}{\big)}\cap {(} N_1 \cup N_2 {)}  \\ & \subseteq &d^s_{m_1}(T_1) \cap d^s_{m_1}(U_1) \cap d^s_{m_2}(T_2) \cap d^s_{m_2}(U_2) \cap {(} N_1 \cup N_2 {)}.
\end{eqnarray*}
\end{minipage}  \vspace{2mm}

In particular, we have $x \in N_1 \cup N_2$. But if $x \in N_1$, then  $x \in d^s_{m_1}(T_1) \cap d^s_{m_1}(U_1) \cap N_1$, contradicting  $(\dagger)_1$. And if $x \in N_2$, then  $x \in d^s_{m_2}(T_2) \cap d^s_{m_2}(U_2) \cap N_2$, contradicting  $(\dagger)_2$. Hence, $N_1\cup N_2 $  is not $n$-s-stationary in $\pk{A}$. \qedhere
\end{enumerate}
 \end{proof}\vspace{1mm} 

Now, recall that an ideal $I$ on $X$ is proper if and only if $X \notin I$. Therefore, if $\mathrm{NsS}_{\kappa}^n(A)$ is proper, we have $\pk{A}\notin \mathrm{NsS}_{\kappa}^n(A)$, meaning that  $\pk{A}$ is $n$-s-stationary in itself. Conversely, if $\pk{A}$ is $n$-s-stationary, then  $\pk{A}\notin \mathrm{NsS}_{\kappa}^n(A)$, ensuring that   $\mathrm{NsS}_{\kappa}^n(A)$ is proper. Thus, Proposition \ref{eqddc} provide us with the analogue to Theorem 3.1 in \cite{B2}.\vspace{1mm}

\begin{theorem}\label{th2} 
For every $n< \omega$, the set $\mathrm{NsS}_{\kappa}^n(A)=\{N \subseteq \pk{A} : N $ is not $n$-s-stationary  $\text{ in } \pk{A} \}$ is a proper ideal on $\pk{A}$  if and only if $\pk{A}$ is $n$-s-stationary in $\pk{A}$. $\square$ 
\end{theorem}\vspace{1mm}

An intriguing question emerges when considering the consistency strength of $n$-stationarity in $\pk{A}$. While we have addressed this question for $n=1$ through Proposition \ref{ct}, the situation becomes more intricate for $n \geq 2$. In the ordinal case, $1$-stationarity only demanded an ordinal with uncountable cofinality. However, for $2$-stationarity, the ordinal $\alpha$ needed to be either inaccessible or the successor of a singular cardinal \cite{B2}. We will conclude this section by demonstrating that, similarly to the one ordinal case, for $\pk{A}$ to exhibit $2$-stationarity in itself, a notable strenghtening of the combinatorial properties of $\kappa$ is required. \vspace{1mm}

 \begin{definition}\label{defm} 
 Suppose $\kappa \subseteq B \subseteq A$, $X \subseteq \pk{A}$ and $Y \subseteq \pk{B}$. Then $X{\restriction}_B := \{ z \cap B \in \pk{B}: z \in X\} $ and $Y^A := \{ z \in \pk{A} : z \cap B \in Y\} $.\vspace{1mm} 
 \end{definition}

Under the above setting, it is easy to see that  $(X {\restriction}_B)^A \supseteq X$ and $(Y^A) {\restriction}_B=Y$.\vspace{1mm}

\begin{prop}[Menas]\label{menas} 
Suppose $\kappa \subseteq B \subseteq A$. Then 
\begin{enumerate}
\item If $S\subseteq \pk{A}$ is stationary in $\pk{A}$, then $S{\restriction}_B$ is stationary in $\pk{B}$.
\item If $S\subseteq \pk{B}$ is stationary in $\pk{B}$, then $S^A$ is stationary in $\pk{A}$.
\end{enumerate}
\end{prop}\vspace{1mm}

Proposition \ref{menas} is well known, and the reader may refer to \cite{J3} for a proof. Given a set $E \subseteq \kappa$, let us define  $E^*:=\{ x \in \pk{A} : |x \cap \kappa| \in E\} \subseteq \pk{A}$.

\begin{lemma}\label{jiu} 
If $\kappa$ is weakly Mahlo and $E \subseteq \{ \mu < \kappa : \mu \text{ is a regular limit  cardinal}\}$ is stationary in $\kappa$. Then $E^*$ is $1$-stationary in $\pk{A}$.
\end{lemma}

\begin{proof} 
Let $T$ be cofinal in $\pk{A}$, we have to prove that $d_0(T) \cap E^* \neq \varnothing$. Construct a transfinite sequence for ordinals  $\alpha < \kappa$, as follows \vspace{2mm}

\quad $x_0 \in T$, 

\quad $x_{\alpha +1} \in T$ is such that $x_{\alpha + 1} \supsetneq x_{\alpha}\cup  \alpha$,

\quad $x_{\alpha} \in T $ is such that  $x_{\alpha} \supsetneq \bigcup_{\gamma <  \alpha} x_{\gamma}$, for $\alpha < \kappa $ limit. \vspace{2mm}

This sequence satisfies the conditions required in Corollary \ref{method}. Then, there is a subsequence $\langle x_{\beta_{\alpha}} : \alpha < \mu \rangle$ such that $x:=  \bigcup_{\alpha <\mu} x_{\beta_{\alpha}}  \in \pk{A}$, $\mu := |x \cap \kappa| $ is a regular limit cardinal, $\mu \in E$,  and  $T \cap \pkld{\mu}{x}$ is cofinal in $ \pkld{\mu}{x}$. That is, $x \in d_0(T)$ and $\mu=|x \cap \kappa| \in E$. Therefore, $x \in d_0(T) \cap E^*$. 
\end{proof} \vspace{1mm} 

Given $n <\omega$, we say that $\kappa$ is $(n+1)$-weakly-Mahlo if and only if the set $W_n:=\{ \mu < \kappa : \mu$ is $n$-weakly-Mahlo$\}$ is stationary in $\kappa$ (See \cite{J4}). Notice that $\kappa$ is $1$-weakly-Mahlo if and only if $\kappa$ is weakly Mahlo. Moreover, for all $n<\omega$, we have that $W_n \subseteq \{ \mu < \kappa : \mu$ is a regular limit  cardinal$\}$. Since elements of $W_n$ are actually cardinals, for $n > 0$ we have,  
\begin{align*}
(W_{n-1})^*\cap C_{\kappa}&=\{ x \in \pk{A} : |x \cap \kappa| \in W_{n-1}\} \cap C_{\kappa} \\
&=\{ x \in \pk{A} : x \cap \kappa \in W_{n-1}\} \\ &= (W_{n-1})^A.
\end{align*}

Recall that if $X \subseteq \kappa$, then $X\subseteq \pk{A}$. Moreover, a set $S\subseteq \kappa$ is stationary in $\kappa$ if and only if $S$ is stationary in $\pk{\kappa}$ (See \cite{J3}). In the following we will talk about $W_n$ being subset of $\kappa$, or  being subset of $\pk{\kappa}$, indistinctly.\vspace{1mm}

\begin{corollary}\label{35} 
Let $n > 0$. If $\kappa$ is $n$-weakly-Mahlo, then $(W_{n-1})^A$ is $1$-stationary in $\pk{A}$.  $\square$ 
\end{corollary}

\begin{prop}\label{nlwm} 
If $\pk{A}$ is $2$-stationary in $\pk{A}$, then $\kappa$ is $n$-weakly-Mahlo for al $n<\omega$. 
\end{prop} 
 
\begin{proof} 
Clearly $\kappa$ is $1$-weakly-Mahlo. Aiming for a contradiction, let $n+1$ be the least number such that  $\kappa$ is  $n$-weakly-Mahlo and $\kappa$ is not $(n+1)$-weakly-Mahlo.\vspace{2mm}

From $\kappa$ not being $(n+1)$-weakly-Mahlo we get that $W_n$ is not stationary in $\kappa$. Therefore, using $(W_n)^A{\restriction}_{\kappa}=W_n$ and Proposition \ref{menas} (1), we have that $(W_n)^A$  is not stationary in $\pk{A}$. Thus, there is some club $C$  in $\pk{A}$ such that $C \cap (W_n)^A = \varnothing$.\vspace{2mm}

From $\kappa$ being $\ n$-weakly-Mahlo and Corollary \ref{35}, we get that $(W_{n-1})^A$ is $1$-stationary in $\pk{A}$. Therefore, by Proposition \ref{new20}, $C \cap (W_{n-1})^A$ is  $1$-stationary in $\pk{A}$, and so by $2$-stationary of $\pk{A}$,  there is $$x \in d_1(C \cap (W_{n-1})^A ) \subseteq d_1(C)  \cap d_1( (W_{n-1})^A).$$

Since $x \in d_1( (W_{n-1})^A)$, we have that $(W_{n-1})^A \cap \pkld{\mu}{x}$ is $1$-stationary in $\pkld{\mu}{x}$. Thus, by Proposition \ref{1stist},  $(W_{n-1})^A \cap \pkld{\mu}{x}$ is stationary in $\pkld{\mu}{x}$. Notice that,
\begin{eqnarray*}
(W_{n-1})^A \cap \pkld{\mu}{x}&=&\{z \in \pk{A} : z \cap \kappa \in W_{n-1} \wedge z \subseteq x \wedge |z|< \mu \} \\ &\subseteq & \{ z \in \pk{x} :  z \cap \kappa \in W_{n-1} \cap \mu \} \\ &= &(W_{n-1}\cap \mu )^x. 
\end{eqnarray*}

Then $(W_{n-1} \cap \mu )^x$ is also stationary in $\pkld{\mu}{x}$. Using $(W_{n-1}\cap \mu )^x{\restriction}_{\mu}=W_{n-1}\cap \mu $ and Proposition \ref{menas} (1), we get that  $W_{n-1} \cap \mu \subseteq \mu$ is stationary in $\pkld{\mu}{\mu}$, and so $W_{n-1} \cap \mu$ is stationary in $\mu$. Then $\mu=|x \cap \kappa|$ is $n$-weakly-Mahlo, and so $x \in (W_{n})^A$. But we also have $x \in d_1(C) \subseteq C$. Hence, we found $x \in C \cap (W_n)^A = \varnothing$, which is impossible. 
\end{proof} 

\section{ A topology for $\pk{A}$}

Now, we direct our attention to establishing a correspondence between $n$-s-stationarity in $\pk{A}$ and a topology for $\pk{A}$. A primary challenge arises due to $\pk{A}$ not being a total order, eliminates the possibility of establishing a natural order topology as a starting point. Nonetheless, the study of $\pk{A}$ has introduced an alternative ordering, distinct from $``\subseteq"$, which has proven more suitable for addressing certain aspects of $\pk{A}$ (See \cite{Car1,Car2,M1}). For every $x,y \in \pk{A}$, define
$$x<^*y \text{ if and only if } x \subseteq y \text{ and } |x|<|y\cap \kappa|.$$
And, for $x,y \in \pk{A}$, we define the interval $(x,y):=\{z \in \pk{A} : x <^* z <^* y\}$.\\

 Notice that if $|y\cap \kappa| > \omega$, then $x,x' <^* y$ if and only if $x\cup x' <^*y $. This is so, because $|x|,|x|' < |y \cap \kappa|$ if and only if $|x \cup x'| < |y \cap \kappa|$. Also, notice that $(y,x) \subseteq \pkld{\mu}{x}$ where $\mu:=|x \cap \kappa|$. For if $y <^* z <^* x$, then $z \subseteq x$ and $|z| < |x \cap \kappa|=\mu$, i.e. $z \in \pkld{\mu}{x}$.\\

The topology we aim to define on $\pk{A}$ is a topology such that its limit points correspond to points exhibiting simultaneous stationary reflection in  $\pk{A}$. More precisely, we want the limit points of a set $S$ to lie whiting  some $d_n^s(S)$ for a give $n$. In any such case, these limit points $x$ should satisfy the condition that $|x \cap \kappa|$ is a regular limit cardinal. With this in mind, and based on our established results while considering the order  $<^*$, we propose the following definition.\vspace{1mm}

\begin{definition}\label{leq} 
For $y \in \pk{A}$ such that $|y \cap \kappa|$ is a regular limit cardinal and $x<^*y$, we define \vspace{-2mm} $$U_x(y):=(x,y)\cup \{y\} \subseteq \pk{A}$$
And, in general, we define $\mathcal{S}_0 \subseteq \mathcal{P}(\pk{A}$) by 
\begin{align*}
\mathcal{S}_0:=\{ \emptyset \} \cup &\{ \{y\} : y \in \pk{A} \text{ and } |y \cap \kappa| \text{ is not a regular limit cardinal} \} \\ \cup &\ \{U_x(y) :  |y\cap \kappa|  \text{ is a regular limit cardinal and } x<^*y\}.
\end{align*} 
\end{definition} \vspace{1mm}

Let us note a basic fact regarding Definition \ref{leq} that we will  constantly use. Let $x,x',y \in \pk{A}$ such that $|y \cap \kappa|$ is a regular limit cardinal, then $x,x' <^* y$ implies  $U_{x\cup x'}(y) \subseteq U_x(y) \cap U_{x'}(y)$. To see that, first notice that $U_{x\cup x'}(y)$ is defined, because $x,x' <^* y$ implies  $x\cup x' <^*y $. Now, let $w \in U_{x\cup x'}(y)$, then $x\cup x' <^* w < ^* y $ or $w=y$. This is,  $x,x' <^* w < ^* y $ or $w=y$, whence, $w \in U_x(y) \cap U_{x'}(y)$.\vspace{1mm}

\begin{definition}\label{leq2} 
Let $\tau_0$  be the topology generated by $\mathcal{S}_0$. This is, the smallest topology containing  $\mathcal{S}_0$.\vspace{1mm}
\end{definition}

\begin{prop}\label{s0b}  
$\mathcal{S}_0$ is a base for the topology $\tau_0$.  
\end{prop} 

\begin{proof} 
To prove that $\mathcal{S}_0$ is a base, we will prove (1) $\pk{A} \subseteq \bigcup \mathcal{S}_0$ and (2) if $z \in V_1 \cap V_2$ for $V_1,V_2 \in \mathcal{S}_0$, then $z \in V$ for some $V \in  \mathcal{S}_0$ such that $V \subseteq V_1 \cap V_2$.\vspace{2mm}

 (1) Let $x \in \pk{A}$. If $|x \cap \kappa|$ is not a regular limit cardinal, then $x \in \{x\} \in \mathcal{S}_0$. If  $|x \cap \kappa|$ is a regular limit cardinal,  then $x \in (z, x) \cup \{x\} = U_{z}(x) \in \mathcal{S}_0$, for any $z$ such that $z \subseteq x$ and  $|z|< |x \cap \kappa|$. Hence, $x \in  \bigcup \mathcal{S}_0$.\vspace{2mm}

(2) Let $z \in V_1 \cap V_2$, for some $V_1,V_2 \in \mathcal{S}_0$. If any of $V_1$ and $V_2$ are equal to $\{x\} \in \mathcal{S}_0$ for some $x\in \pk{A}$, then $V_1\cap V_2 = \{x\}$.  So that if we choose $V:=\{x\} \in \mathcal{S}_0$  we   trivially have  that  $V \subseteq V_1 \cap V_2$. \vspace{2mm}

If neither $V_1$ nor $V_2$ are equal to $\{x\} \in \mathcal{S}_0$ for any $x\in \pk{A}$. Then $V_1=U_x(y)$ and $V_2=U_{x'}(y')$, for some $x,x',y,y' \in \pk{A}$ such that $|y \cap \kappa|,|y' \cap \kappa|$ are regular limit cardinals. Now, $z \in U_x(y) \cap U_{x'}(y')$, this is, $z \in (x,y) \cup \{y\}$ and $ z \in (x',y') \cup \{y'\} $. Consider two cases:\vspace{2mm}
\quad\\
(i) $z \in \{y\} \cup \{y'\}$: W.l.g., let $z=y$. Then from  $ z \in (x',y') \cup \{y'\} $, we have that either  $z=y'$ or $z  \in (x',y') $.
\begin{itemize}
\item[a)]If $z =y'$: Then $y=y'=z$, and so we have $x,x' <^* y=z$. Then, $$z \in U_{x\cup x'}(y) \subseteq U_x(y) \cap U_{x'}(y)=U_x(y) \cap U_{x'}(y').$$ 
\item[b)] If $z \in (x',y')$: Then $x,x' <^*z=y <^*y' $, and so $z\in U_{x\cup x'}(y)\subseteq U_x(y) \cap U_{x'}(y)$. Now, notice that $U_{x}(y) \subseteq U_{x}(y') $, then we have $$z\in U_{x\cup x'}(y)\subseteq U_x(y) \cap U_{x'}(y) \subseteq U_x(y) \cap U_{x'}(y').$$
\end{itemize}\vspace{2mm}
(ii) $z \notin \{y\} \cup \{y'\}$: Then $z \in (x,y) \cap (x',y')$, and so $x,x' <^*z<^*y,y'$. First, suppose that $|z \cap \kappa|$ is not a regular limit cardinal. Then, $$z \in \{z\} \subseteq U_x(y) \cap U_{x'}(y')$$ and since $\{z\} \in \mathcal{S}_0$,  we are done.  Now, suppose $|z \cap \kappa|$ is a regular limit cardinal. From $x,x' <^*z$, we have $z \in U_{x\cup x'}(z) \subseteq U_x(z) \cap U_{x'}(z)$. And  from $z<^*y,y'$, we have, $U_x(z)\subseteq U_x(y)$ and $U_{x'}(z) \subseteq U_{x'}(y')$. Therefore,  $$z \in U_{x\cup x'}(z) \subseteq U_x(z) \cap U_{x'}(z)\subseteq U_x(y) \cap U_{x'}(y').\qedhere$$  
\end{proof}\vspace{1mm} 

We refer back to our preferred notation for Cantor's Derivative operator on subsets of $\pk{A}$. Given a topology $\tau$ in $\pk{A}$ and $X\subseteq \pk{A}$, we denoted $\partial_{\tau}(X):= \{x \in \pk{A} : x $ is a limit point of $X$ in $\tau \}$.\vspace{1mm}

\begin{prop}\label{s0=d0} 
For any $S \subseteq \pk{A}$, $d^s_0(S)=\partial_{\tau_0}(S)$.  
\end{prop} 

\begin{proof} 
($\subseteq$) Suppose  $x \in d_0^s(S)$. Let $U$ be a basic open set containing $x$, clearly $U \neq \{x\}$ because $\mu:=|x \cap \kappa|$  is a regular limit cardinal. Then $x \in U= U_z(y)$ for some $y,z \in \pk{A}$ such that $|y \cap \kappa|$ is a regular limit  cardinal and $z <^*x$. In particular $x \in U_z(x) \subseteq  U_z(y)$. Since $z<^*x$, then $z \in \pkld{\mu}{x}$. Also, since $\mu$ is a limit cardinal,  $z \cup |z|^+ \subseteq x$ and $|z \cup |z|^+ |=|z|^+ < \mu$.  Then from  $x \in d_0^s(S)$, we get that there is $z_0 \in S \cap \pkld{\mu}{x}$ such that $z \cup |z|^+  \subseteq z_0$. This is $z \subseteq z_0$ and $|z| < |z|^+=|(z \cup |z|^+) \cap \kappa| \leq |z_0 \cap \kappa|$. Whence $z_0 \in (z,y)$, and so $z_0 \in U_z(y) \cap (S \setminus \{x\})$. 
\vspace{-1mm}

($\supseteq$) Suppose $x \in \partial_{\tau_0}(S)$, this is, $x$ is a limit point of $S$ in the topology $\tau_0$. Notice that $\mu:=|x \cap \kappa|$ is a regular limit cardinal, otherwise $\{x\} \in \tau_0$, and this would contradict that  $x$ is a limit point.  Let $y \in  \pkld{\mu}{x}$, then $y \subseteq x$ and $|y|<\mu= |x \cap \kappa|$, and so $x \in U_{y}(x) \in \mathcal{S}_0$. From the definition of $x$ being limit point, we get that  $U_{y}(x) \cap (S\setminus \{x\}) \neq \varnothing$. Then, there is some $z \in (y,x) \cap S$. Thus, $ z \in S \cap \pkld{\mu}{x}$ and $y \subseteq z$. Therefore $S \cap \pkld{\mu}{x}$ is  cofinal in $\pkld{\mu}{x}$. 
\end{proof}\vspace{2mm}

Therefore, for the case of $n=0$, our constructed topology emulates the behaviour of the order topology in ordinals concerning $0$-s-stationary reflection. Consequently, we select $\tau_0$ as the initial topology for our intended sequence of topologies $\langle \tau_o, \tau_1 , \dots \rangle $ on $\pk{A}$. As in \cite{B2}, given $\tau_n$ for $n <\omega$ we define $\partial_{\tau_n}$ to be the Cantor's derivative operator $\partial_{\tau_n}(S):=\{ x \in \pk{A} : x $ is a limit point of $S$ in the topology $\tau_n \}$. Set $\mathcal{B}_0:=\mathcal{S}_0$ and $\mathcal{B}_{n+1}:=\mathcal{B}_n \cup \{ \partial_{\tau_n}(S) : S \subseteq \pk{A}\}$. Finally, define $\tau_{n+1}$ to be the topology generated by $\mathcal{B}_{n+1}$. This is $\tau_{n+1}:= \langle \mathcal{B}_{n+1} \rangle$.\vspace{1mm}

\begin{definition}\label{subas} 
Given $n< \omega$,  define  
\begin{align*}
\mathcal{S}_{n+1}:= \{U_z(y) \cap  \mathop{\bigcap}\limits_{i=1}^{m} d_n^s(T_i) : \ &T_i \subseteq \pk{A}, \ 0\leq m < \omega, \text{ and  } z,y \in \pk{A}, \\ & |y \cap \kappa| \text{ a regular limit cardinal, and }z<^*y.\}.
\end{align*}
\end{definition}\vspace{1mm}

\begin{prop}\label{equuddc} 
For every $n< \omega$, $\mathcal{S}_{n+1}$ is a base for a topology in $\pk{A}$.
\end{prop}
 
\begin{proof} 
We will prove (1) $\bigcup \mathcal{S}_{n+1}= \pk{A}$ and (2) for every $x \in \pk{A}$ and $U,V \in \mathcal{S}_{n+1}$, if $x \in U\cap V$, there  is $W \in \mathcal{S}_{n+1}$ such that $x \in W \subseteq U \cap V$. \\
\quad\\
(1) Since $\mathcal{S}_0$ is a base, in particular, $\pk{A} = \bigcup \mathcal{S}_0$. Moreover, from Definition \ref{subas}, we get  that $\mathcal{S}_0 \subseteq \mathcal{S}_{n+1}$. Then  $\pk{A} =   \bigcup \mathcal{S}_0 \subseteq \bigcup \mathcal{S}_{n+1} \subseteq \pk{A}$.\\
\quad\\
(2) Let $x \in \pk{A}$, and take $U,V \in \mathcal{S}_{n+1}$ such that $x \in U\cap V$. Then, there are $U_{z'}(y'),U_{z''}(y'')\in \mathcal{S}_0$ such that 

\begin{center} 
$U = U_{z'}(y')\ \cap \ \mathop{\bigcap}\limits_{i=1}^{m_1} d_n^s(T^1_i)$, \ where \ $T^1_i \subseteq \pk{A}$ \ for \ $i\leq m_1<\omega$. 
\end{center}  

 \begin{center} 
 $V= U_{z''}(y'') \ \cap \ \mathop{\bigcap}\limits_{i=1}^{m_2} d_n^s(T^2_i)$, \ where \ $T^2_i \subseteq \pk{A}$ \ for \ $i \leq m_2<\omega$.
 \end{center} 
Setting $m:=m_1+m_2$, $T_i:=T^1_i$ for $1 \leq i \leq m_1$, and $T_{i+m_1}:=T^2_i$ for $1 \leq i \leq m_2$, we obtain
\begin{equation}
x \in U \cap V = U_{z'}(y') \ \cap \ U_{z''}(y'') \ \cap \ \mathop{\bigcap}\limits_{i=1}^{m} d_n^s(T_i).
\end{equation}
\quad\\
If $m=0$, then $U, V \in \mathcal{S}_0$, and using Proposition \ref{s0b}, we get $W\in \mathcal{S}_0\subseteq \mathcal{S}_{n+1}$ such that $x \in W \subseteq U \cap V.$\vspace{2mm}
\quad\\ 
If $m \geq1$, then $x \in d^s_{n}(T_1)$, and so  $|x \cap \kappa|$ is a regular limit cardinal. Note that $x \in U_{z'}(y') \cap U_{z''}(y'')$ implies $z',z'' <^* x <^* y,y'$. Thus, if $z:=z' \cup z'' $ and $y:= \min\{y',y''\}$, we have 

\begin{equation}
x \in U_z(x) \subseteq U_z(y) \subseteq U_{z'}(y) \cap U_{z''}(y)\subseteq U_{z'}(y') \cap U_{z''}(y'').
\end{equation}
\quad\\
Therefore, if $W:=  U_{z}(x) \cap \mathop{\bigcap}_{i=1}^{m} d_n^s(T_i)$, we clarly have $W\in \mathcal{S}_{n+1}$. Moreover, by equations (1) and (2), we get $x \in W = U_{z}(x) \ \cap \ \mathop{\bigcap}\limits_{i=1}^{m} d_n^s(T_i) \subseteq U \cap V. $ 
\end{proof}\vspace{1mm} 

\begin{definition}\label{hatsr}  
Let $n< \omega$. Define $\tau_n^*$ to be the topology generated by $\mathcal{S}_n$. Also, for any $S \subseteq \pk{A}$ let us write $\partial_n(S)$ for $ \partial_{\tau^*_n}(S)$.\vspace{1mm} 
\end{definition}

\begin{prop}\label{equuddc}  
If $n< \omega$ and  $S \subseteq \pk{A}$, then $d^s_n(S)= \partial_n(S)$. 
\end{prop}

\begin{proof} 
Proceed by induction on $n$. The case $n=0$ is precisely Proposition \ref{s0=d0}. Suppose the proposition holds for $n$ and let us prove it for $n+1$.\vspace{2mm}
  \quad\\
  ($\subseteq$) Let $x \in d^s_{n+1}(S)$. Then $\mu:=|x\cap \kappa|$ is a regular limit cardinal and $S \cap \pkld{\mu}{x}$ and $\pkld{\mu}{x}$ are $(n+1)$-s-stationary in $\pkld{\mu}{x}$. Let $U \in \mathcal{S}_{n+1}$ such that $x \in U$, we aim to prove that $U \cap (S \setminus \{x\})\neq \varnothing$. Notice that, $U=U_{z}(x) \cap \mathop{\bigcap}_{i=1}^{m} d_n^s(T_i)$ for some $T_i \subseteq \pk{A}$, $m < \omega$ and $U_z(y) \in \mathcal{S}_0$. Consider two cases\vspace{2mm}
\begin{itemize}
\item[(i)] ($m=0$) : Then $U \in \mathcal{S}_{0}$. By Proposition \ref{s0=d0}, we have that $x \in d^s_{n+1}(S) \subseteq d^s_{0}(S)=$ $ \partial_0(S)$. Thus, $x$ is a limit point of $S$ in $\tau_0$, and so $U \cap (S \setminus \{x\}) \neq \varnothing$.\vspace{2mm}

\item[(ii)] ($m\geq1$) : Then $x \in \mathop{\bigcap}_{i=1}^{m} d_n^s(T_i)$, and so $T_i \cap \pkld{\mu}{x} $ is $n$-s-stationary in $\pkld{\mu}{x} $ for all $i \in \{1, \dots,m\}$. Since $S \cap \pkld{\mu}{x}$ is  $(n+1)$-s-stationary in $\pkld{\mu}{x}$,  by Proposition \ref{eqc} (2), we get that $S \cap \pkld{\mu}{x} \cap \bigcap_{i=1}^{m} d_n^s(\pq{T_i \cap \pkld{\mu}{x}})$ is $n$-s-stationary in $ \pkld{\mu}{x}$, and in particular cofinal.\vspace{2mm}

Let $w  \in U_z(x) \subseteq U_z(y)$, then $w \in \pkld{\mu}{x}$. Thus, there is $w' \in S \cap  \pkld{\mu}{x}  \cap\mathop{\bigcap}_{i=1}^{m} d_n^s(T_i)$  such that $z \subseteq w \cup |w|^+ \subseteq w'$. Notice that $w \cap \kappa \subseteq (w \cup |w|^+) \cap \kappa \subseteq w' \cap \kappa$, and $|z|\leq |w| < |w|^+= | (w \cup |w|^+)\cap \kappa | \leq |w' \cap \kappa|$. Then $w' \in  U_z(y)$ and so

\begin{minipage}{\linewidth}    
\begin{eqnarray*} 
w' \in U_z(y)  \cap\mathop{\bigcap}_{i=1}^{m} d_n^s(T_i) \cap S =  U \cap S.
\end{eqnarray*}
\end{minipage}\vspace{2mm} 
Finally, notice that $w' \neq x$, because $w' \in \pkld{\mu}{x}$. Hence, $U \cap (S \setminus \{x\}) \neq \varnothing$.
\end{itemize}\vspace{2mm}
 ($\supseteq$) Let $x \in \partial_{n+1}(S)$. Then $x$ is limit point of $S$ in $\tau_{n+1}$. If $|x \cap \kappa|$ is not a regular limit cardinal, then $\{x\} \in \mathcal{S}_0 \subseteq \mathcal{S}_{n+1}$ and so $x$ cannot be a limit point of $S$ in $\tau_{n+1}$. Then $\mu:=|x \cap \kappa|$ is a regular limit cardinal. Now, consider two cases;
\begin{itemize}
\item[(i)] $\pkld{\mu}{x}$ is not $(n+1)$-s-stationary in $\pkld{\mu}{x}$: Then, there are $T_1,T_2 \subseteq \pkld{\mu}{x}$ that are $m$-s-stationary in $\pkld{\mu}{x}$ for some $m \leq n$, and such that $d^s_m(T_1) \cap d^s_m(T_2) \cap \pkld{\mu}{x} = \varnothing$. Recall that $U_z(x) \subseteq \pkld{\mu}{x} \cup \{x\}$, moreover $T_1,T_2$ are in particular subsets of $\pk{A}$, then $x \in d^s_m(T_1) \cap d^s_m(T_2)$. Thus,
\begin{minipage}{\linewidth}    
\begin{eqnarray*}
x &\in &U_z(x)  \cap d^s_m(T_1) \cap d^s_m(T_2) \\ & \subseteq & {\big(} \pq{\pkld{\mu}{x} \cup \{x\} }{\big)} \cap d^s_m(T_1) \cap d^s_m(T_2) \\ &=& {\big(}\pq{ \pkld{\mu}{x} \cap d^s_m(T_1) \cap d^s_m(T_2)} {\big)} \cup {\big(} \pq{\{x\} \cap d^s_m(T_1) \cap d^s_m(T_2)}  {\big)} \\ &=& \varnothing \cup \{x\} \\ &=& \{x\}.
\end{eqnarray*}
\end{minipage}\vspace{2mm} 
 This is, $ \{x\}= U_z(x)  \cap d^s_m(T_1) \cap d^s_m(T_2) \in \mathcal{S}_{n+1}$. And so in this case $x$ cannot be a limit point of $S$ in $\tau_{n+1}$.\vspace{2mm}

\item[(ii)] If $\pkld{\mu}{x}$ is $(n+1)$-s-stationary in $\pkld{\mu}{x}$: We aim to prove that $x \in d^s_{n+1}(S)$, namely, $S \cap  \pkld{\mu}{x}$ is $(n+1)$-s-stationary in $\pkld{\mu}{x}$. So let $m< n+1$ and $T_1, T_2 \subseteq \pkld{\mu}{x} $ be $m$-s-stationary in $\pkld{\mu}{x} $, notice that $x \in d^s_m(T_1) \cap d^s_m(T_2)$. Now, pick some $z<^*x$, then $x \in  U_z(x) \cap  d^s_m(T_1) \cap d^s_m(T_2) \in \mathcal{S}_{n+1}$. But $x$ is a limit point of $S$ in $\tau_{n+1}$. Then, 

\begin{minipage}{\linewidth}    
\begin{eqnarray*}
\varnothing & \neq  &d^s_m(T_1) \cap d^s_m(T_2) \cap U_z(x)  \cap {(} S \setminus \{x\}  {)} \\ & \subseteq & d^s_m(T_1) \cap d^s_m(T_2) \cap \big(\pkld{\mu}{x} \cup \{x\}\big)\cap {(} S \setminus \{x\}  {)} \\ & \subseteq & d^s_m(T_1) \cap d^s_m(T_2)\cap S \cap  \pkld{\mu}{x}
\end{eqnarray*}
\end{minipage}\vspace{2mm} 
This is, $S \cap  \pkld{\mu}{x}$ is $(n+1)$-s-stationary in $\pkld{\mu}{x}$. Hence, $x \in d^s_{n+1}(S)$. \qedhere
\end{itemize} 
\end{proof}\vspace{1mm} 

 \begin{corollary}
 Suppose $S,T \subseteq \pk{A}$ and $m \leq n < \omega$. Then $d_m^s(S) \cap d_n^s(T) = d_n^s(d_m^s(S) \cap T ).$
 \end{corollary}
 
\begin{proof} 
$\subseteq$) Let $x \in d_m^s(S) \cap d_n^s(T) $. Then by Proposition \ref{equuddc}, $x$ is a limit point of $T$ in the topology $\tau_n$. Moreover $x \in d_m^s(S) \cap U_y(x) \in \tau_m \subseteq \tau_n$ for some $y<^*x$. To prove that $x \in d_n^s(d_m^s(S) \cap T )$, we will prove that $x$ is a limit point of $d_m^s(S) \cap T$ in $\tau_n$. Let $U \in \tau_n$ be such that $x \in U$, then $x \in U \cap d_m^s(S) \cap U_y(x) \in \tau_n$. Then $$\varnothing \neq  U \cap d_m^s(S) \cap U_y(x)  \cap (T \setminus \{x\}) \subseteq   U \cap ((d_m^s(S) \cap T) \setminus \{x\}).$$ This is, $x $ is a limit point of f $d_m^s(S) \cap T$ in $\tau_n$, and so $x \in \partial_n(d_m^s(S) \cap T )=d_n^s(d_m^s(S) \cap T )$\vspace{2mm}
 \quad\\
 ($\supseteq$) By Lemma \ref{lem1}, $ d_n^s(d_m^s(S) \cap T ) \subseteq  d_n^s(d_m^s(S)) \cap d_n^s(T) \subseteq d_m^s(S) \cap d_n^s(T)$. 
 \end{proof}\vspace{1mm}

It is a standard fact that for a set $X$ and a collection $\mathcal{B}\subseteq \mathcal{P}(X)$ such that $\bigcup \mathcal{B}=X$, the topology generated by $\mathcal{B}$, denoted as $\langle \mathcal{B} \rangle$, consists of all unions of finite intersections of elements from $\mathcal{B}$. Define $\mathcal{B}_0^*:=\mathcal{S}_0$ and for all $n \leq \omega$, $\mathcal{B}_{n+1}^*:=\mathcal{B}_n^* \cup \{d^s_{n}(S) :  S \subseteq \pk{A} \}$. Then due to the way we defined  $ \mathcal{S}_n$ for  $n<\omega$, it follows immediately that $\tau_n^*=\langle \mathcal{S}_{n} \rangle=\langle \mathcal{B}_{n}^* \rangle$  for  all $n<\omega$.\\

We previously denoted the derived topologies on $\pk{A}$, as $\tau_{n}= \langle \mathcal{B}_{n} \rangle$, where $\mathcal{B}_0=\mathcal{S}_0$ and $\mathcal{B}_{n+1}=\mathcal{B}_n \cup \{ \partial_{\tau_n}(S) : S \subseteq \pk{A}\}$. We claim that, $\tau_n=\tau_n^*$ for all $n<\omega$. We can see this by  an easy induction on $n$. For $n=0$ both definitions are the same $\tau_0=\langle \mathcal{S}_{0} \rangle=\tau_0^*$. Suppose that $\tau_n=\tau_n^*$. Then $\partial_{\tau_n}(S)=\partial_{\tau_n^*}(S)=\partial_{n}(S)$ and $\langle \mathcal{B}_{n}^* \rangle=\langle \mathcal{B}_{n} \rangle$. By previous paragraph and Proposition \ref{equuddc}, we get  
\begin{minipage}{\linewidth}    
\begin{eqnarray*}
\tau_{n+1}^*&=&\langle \mathcal{B}^*_n  \cup \{ d_n^s(S) : S \subseteq \pk{A}\} \rangle= \langle \mathcal{B}^*_n \cup \{ \partial_{n} (S) : S \subseteq \pk{A}\} \rangle  \\  & = & \langle \mathcal{B}_n \cup \{ \partial_{n}(S) : S \subseteq \pk{A}\} \rangle= \langle  \langle \mathcal{B}_n \cup \{ \partial_{\tau_n}(S) : S \subseteq \pk{A}\} \rangle=\tau_{n+1}.
\end{eqnarray*}
\end{minipage}\vspace{2mm} 

 Therefore, for $S\subseteq \pk{A}$ and for the derived topology $\tau_n$ on $\pk{A}$, we have $d_{n}^s(S)=\partial_{\tau_n}(S)$. Consequently, we have established our desired correspondence between $n$-s-stationary sets of $\pk{A}$ and the derived topologies on $\pk{A}$ obtained by means of Cantor's derivative operator. Hence, we arrived to the analogue of Theorem 2.11 in \cite{B2}.\vspace{2mm}

 \begin{theorem}\label{t3}
 For every $n < \omega$, a point $x \in \pk{A}$ is not isolated in the $\tau_n$ topology on $\pk{A}$ if and only if $x$ is such that $\mu:=|x \cap \kappa|$ is a regular limit cardinal and $\pkld{\mu}{x}$ is $n$-s-stationary in $\pkld{\mu}{x}$. Thus, $\mathcal{B}_n$ generates a non-discrete topology on $\pk{A}$ if and only if some $x$ is such that $\mu:=|x \cap \kappa|$ is a regular limit cardinal and $\pkld{\mu}{x}$ is $n$-s-stationary in $\pkld{\mu}{x}$. \vspace{1mm}
 \end{theorem}
 
 \begin{proof} 
 Notice that, $x\in \pk{A}$ is not isolated in the  in the $\tau_n$ topology on $\pk{A}$ if and only if $x$ is a limit point of $\pk{A}$ if and only if $x \in \partial_{\tau_n}(\pk{A})$. Moreover, $x \in \partial_{\tau_n}(\pk{A})= d_n^s(\pk{A})$ if and only if  $\mu:=|x \cap \kappa|$ is a regular limit cardinal and $\pkld{\mu}{x}$ is $n$-s-stationary in $\pkld{\mu}{x}$. 
 \end{proof}\vspace{1mm} 

\section{A sufficient condition for $n$-stationarity in $\pk{A}$}

Let $A,B$ be sets such that $\kappa \subseteq A,B$, and $|A|=|B|$. It is easy to show that, for any bijection $g: A \rightarrow B$, the function $g^*:\langle\pk{A},\subseteq\rangle\rightarrow\langle\pk{B},\subseteq\rangle$\ defined by $g^*(x):=g[x]=\{g(y) : y \in x\}$ is an isomorphism. Moreover, $g^*$ preserves cofinal, stationary and club sets (See \cite{J2}). We will refer to  $g^*$  as \textit{the isomorphism induced by} $g$. We will show in Lemma \ref{funt} that, there is also an isomorphism preserving higher stationary sets. \vspace{2mm}

\begin{lemma} 
Let $\kappa$ be a limit cardinal such that $\kappa \subseteq A,B$, and let $g: A \rightarrow B$ be a bijection. Then $C^g_{\kappa}:=\{x \in \pk{A}: g[x]\cap \kappa \subseteq x \cap \kappa \}$ is a club in $\pk{A}$.
\end{lemma}

\begin{proof}  
(i) $C^g_{\kappa}$ is unbounded: Let $y_0 \in \pk{A}$. Given $ y_n \in \pk{A}$, since $|g[y_n]| = |y_n| < \kappa$ and $g[y_n] \cap \kappa \subseteq \kappa$, then $g[y_n] \cap \kappa \in \pk{A}$. Thus, let $y_{n+1} := y_n \cup (g[y_n] \cap \kappa) \in \pk{A}$. Define $z := \bigcup_{n < \omega} y_n$. Clearly $y_0 \subseteq z \in \pk{A}$. Besides, $z \in C_{\kappa}^g$, because  $g[z] \cap \kappa \subseteq z \cap \kappa$ follows from 
\begin{align*} 
g[\bigcup_{n<\omega}y_n] \cap \kappa & =  \big( \bigcup_{n<\omega}g[y_n] \big) \cap \kappa  =  \big( \bigcup_{n<\omega}g[y_n]  \cap \kappa \big) \cap \kappa \subseteq  \big( \bigcup_{n<\omega}y_{n+1} \big) \cap \kappa. 
\end{align*}
(ii) $C^g_{\kappa}$ is closed: Let $\langle x_{\alpha} : \alpha < \gamma \rangle$ be a sequence of elements of $C^g_{\kappa}$.  $\bigcup_{\alpha< \gamma}x_{\alpha} \in C^g_{\kappa}$, follows from  
\begin{align*} 
g[\bigcup_{\alpha< \gamma}x_{\alpha}] \cap \kappa =\big( \bigcup_{\alpha< \gamma}g[x_{\alpha}] \big) \cap \kappa =  \bigcup_{\alpha< \gamma} (g[x_{\alpha}] \cap \kappa) \subseteq  \bigcup_{\alpha< \gamma} (x_{\alpha} \cap \kappa)= \big( \bigcup_{\alpha< \gamma} x_{\alpha} \big)\cap \kappa. 
\end{align*}
\end{proof} 

\begin{lemma}\label{funt}  
Let $A,B$ be sets such that $\kappa \subseteq A,B$, and let $g: A \rightarrow B$ be a bijection such that $g(\gamma)=\gamma$ for all $\gamma \in c(\kappa):=\{ \gamma < \kappa : \gamma $ is a cardinal$\}$. Then for all $n < \omega$, the induced isomorphism  $f:=g^*: \langle \pk{A} , \subseteq \rangle \rightarrow \langle \pk{B} , \subseteq \rangle $ preserves $n$-s-stationary sets. This is, $S$ is $n$-s-stationary in $ \pk{A}$ if and only if $f[S]$ is $n$-s-stationary in $ \pk{B}$.
\end{lemma}

\begin{proof} 
Proceed by induction on $n$. The case $n=0$ follows from the fact that $g^*$ preserves cofinal sets. So, let $n\geq1$ and assume the theorem holds for all $m < n$. Let $S\subseteq \pk{A}$ be $n$-s-stationary in $\pk{A}$ and let $C:= C_{\kappa} \cap C^g_{\kappa}$. Clearly $C$ is a club in $\pk{A}$, then by Proposition \ref{new2}, $S\cap C$ is $n$-s-stationary in $\pk{A}$. Take $m<n$ and   $T_1,T_2 \subseteq \pk{B}$ $m$-s-stationary in $\pk{B}$. By induction Hypothesis, $f^{-1}[T_1],f^{-1}[T_2] \subseteq \pk{A}$ are $m$-s-stationary in $\pk{A}$. Then, there is $x \in S \cap C$ such that $\mu := |x \cap \kappa|$ is a regular limit cardinal and $f^{-1}[T_1]\cap \pkld{\mu}{x}$,  $f^{-1}[T_2]\cap \pkld{\mu}{x}$ are $m$-s-stationary in $\pkld{\mu}{x}$. Because $x \in C\subseteq C_{\kappa}$, we actually know that $\mu= |x \cap \kappa|=x \cap \kappa$.\vspace{2mm}

Using induction hypothesis again,  we have that, for $i=1,2$, the set  $$f[f^{-1}[T_i]\cap \pkld{\mu}{x}]=f[f^{-1}[T_i]]\cap f[\pkld{\mu}{x}]=T_i \cap\pkld{\mu}{f(x)}$$ is $m$-s-stationary in $\pkld{\mu}{f(x)}$. Since $x \in  C\subseteq C^g_{\kappa}$, then  $|f(x) \cap \kappa|=|g[x] \cap \kappa| \leq | x\cap \kappa |= \mu$. Moreover, as  $x\cap \kappa$ is a regular limit cardinal, we have $|c(x\cap \kappa)|=|x\cap \kappa|$. Finally, by definition of $g$,  $\gamma \in  c(x\cap \kappa) \subseteq x$ implies $\gamma = g(\gamma) \in g[x] \cap \kappa= f(x) \cap \kappa$. Thus, $g[c(x\cap \kappa)]\subseteq f(x)\cap \kappa$. Combining all of the above, we obtain
 $$\mu = |x \cap \kappa|= |c(x\cap \kappa)|=|g[c(x\cap \kappa)]|\leq|f(x)\cap \kappa|\leq \mu.$$

Therefore, $f(x)\in f[S]$ is such that $\mu=|f(x)\cap \kappa|$ is a regular limit cardinal and $T_i \cap\pkld{\mu}{f(x)}$ is $m$-s-stationary in $\pkld{\mu}{f(x)}$ for $i=1,2$. This is, $f[S]\subseteq \pk{B}$ is $n$-s-stationary in $\pk{B}$. The converse is analogous. 
\end{proof} \vspace{1mm}

\begin{lemma}\label{funt2}   
Let $A,B$ be sets such that $\kappa \subseteq A,B$, and let $g: A \rightarrow B$ be a bijection such that $g(\gamma)=\gamma$ for all $\gamma \in c(\kappa)$. Then, for all $n < \omega$, the induced isomorphism $g^*: \langle \pk{A} , \subseteq \rangle \rightarrow \langle \pk{B} , \subseteq \rangle $ preserves $n$-stationary sets. 
\end{lemma}

\begin{proof} 
Similar to proof of Lemma \ref{funt}. Using Proposition \ref{new20} to prove  that $S\cap  C_{\kappa} \cap C^g_{\kappa}$ is $n$-stationary in $\pk{A}$, and taking  only one set $T\subseteq \pk{B}$ that is $m$-stationary in $\pk{B}$.
\end{proof} \vspace{1mm}

\begin{prop}  
Let $A,B$ be sets such that $\kappa \subseteq A,B$ and $|A|=|B|$. Then, there is an isomorphism $f: \langle \pk{A} , \subseteq \rangle \rightarrow \langle \pk{B} , \subseteq \rangle $ preserving cofinal, stationary, club, $n$-stationary and $n$-s-stationary sets.
\end{prop}

\begin{proof} 
By Lemma \ref{funt} and Lemma \ref{funt2}, it is enough to prove that there is a bijection $g: A \rightarrow B$ such that $g(\gamma):=\gamma$ for all $\gamma \in c(\kappa)$. So, let us prove that: Since $\kappa$ is a regular limit cardinal, $|c(\kappa)|= |\kappa \setminus c(\kappa)|=\kappa$. Then $|A|=|A \setminus c(\kappa)|=|B \setminus c(\kappa)|=|B|$, and so there is a bijection $h:  A\setminus c(\kappa) \rightarrow B\setminus c(\kappa)$. Now, define $g: A \rightarrow B$, by $g(\gamma):=\gamma$ for all $\gamma \in c(\kappa)$ and $g(x):=h(x)$ for all $x \in A\setminus c(\kappa)$.
\end{proof} \vspace{1mm} 

 Thus, if $|A|= \lambda \geq \kappa$, the study of $\pk{A}$ is equivalent to that of $\pk{\lambda}$ (usually written as $\pkl$). In this section we will first  show a sufficient condition to have that $\pk{\kappa}$ is $n$-s-stationary and $n$-stationary  in $\pk{\kappa}$ for some $n< \omega$ (Theorem \ref{inds}). And second, a sufficient condition to have that $\pkl$ is $n$-s-stationary and $n$-stationary  in $\pkl$ for all $n< \omega$ (Theorem \ref{lsc}). The second condition arrived as a more general version of a Proposition originally stated by Sakai et al. in \cite{S1}. Consequently, the proof of this second condition also provides, as a special case, a proof of Sakai et al.'s original proposition. \vspace{2mm}

\begin{prop}\label{p320bg}  
Let $\kappa$ be a regular limit cardinal. Then, 
\begin{itemize}
\item[(1)] The formula $\varphi_n(S) : ``S \subseteq \pk{\kappa}$ is $n$-s-stationary in $\pk{\kappa}"$ is  $\Pi^1_n $   over $\langle V_{\kappa}, \in , S \rangle$.
 \item[(2)] Let $x \in  \pk{\kappa}$ be such that $|x \cap \kappa|$ is a regular limit cardinal. Then  $\varphi'_n(T,x) : ``T  \cap \pkld{|x \cap \kappa|}{x}$ is $n$-s-stationary in $\pkld{|x \cap \kappa|}{x}"$ is a $\Pi^1_0$ sentence over $\langle V_{\kappa}, \in  \rangle$, in the parameters $T,x$.
\end{itemize}
\end{prop}
  
\begin{proof} 
First, notice that $\pk{\kappa} \in V_{\kappa+1} \setminus V_{\kappa}$ and $\pkld{|x \cap \kappa|}{x} \in V_{\kappa}$ for all $x \in \pk{\kappa}$. Then $\mathcal{P}(\pk{\kappa}) \subseteq V_{\kappa+1}$, this is, subsets of $\pk{\kappa}$ are elements of $V_{\kappa+1}$. Also, notice that, $y \in \pk{\kappa}$ if and only if  $\langle V_{\kappa}, \in \rangle \models \exists \alpha ( OR(\alpha) \wedge  y \subseteq \alpha )$. Therefore $``y \in \pk{\kappa}"$   is $\Pi_2^0$, and so $\Pi_0^1$ over  $\langle V_{\kappa}, \in \rangle$. Similarly, $Reg(\alpha):`` \alpha$ is a regular limit cardinal$"$ is $\Pi_0^1$ over  $\langle V_{\kappa}, \in \rangle$. We will now prove both parts of the lemma simultaneously by induction on $n$. \vspace{2mm}
\quad\\
($n=0$) : (1)  $S \subseteq \pk{\kappa}$ is $0$-s-stationary  in $\pk{\kappa}$ if and only if $\varphi_0 (S) $, where, 
\begin{alignat*}{2}
\varphi_0 (S) \ :  \ \forall  y  (  y  \in \pk{\kappa} \rightarrow \exists z \in S \ (y \subseteq z) ).
\end{alignat*}
 
Notice that $y$ ranges over all possible elements of  $\pk{\kappa} \subseteq V_{\kappa}$. Then $y$ is a first-order variable over $\langle V_{\kappa}, \in , S \rangle$, and consequently  $\varphi_0 (S) $  is  $\Pi_0^1$ over $\langle V_{\kappa}, \in , S \rangle$.\\

(2) Let $x \in \pk{\kappa}$ be such that $|x \cap \kappa|$ is a regular limit cardinal. Now,  $T \cap  \pkld{|x \cap \kappa|}{x}$ is  $0$-stationary in $\pkld{|x \cap \kappa|}{x}$ if and only if $\varphi_0'(T , x )$, where 
\begin{alignat*}{2}
\varphi_0'(T,  x ) \ : \ \forall  y  (  y \in  \pkld{|x \cap \kappa|}{x} \rightarrow \exists z \in T  \cap  \pkld{|x \cap \kappa|}{x} (y \subseteq z)).
\end{alignat*}

Notice that $y$  ranges over all possible subsets of  $ \pkld {|x \cap \kappa|}{x} \in V_{\kappa}$. Then $y$ is a first-order variable over $\langle V_{\kappa}, \in \rangle$, and consequently  $\varphi_0'(T,x )$   is a $\Pi_0^1$ statement over $\langle V_{\kappa}, \in , S \rangle$, in the parameters $T$ and $x$.

\quad\\
$(n>m)$: (1) Suppose now that for all $m< n$, $ \varphi_{m} (S):``S \subseteq \pk{\kappa}$ is $m$-s-stationary in $\pk{\kappa}"$ is  $\Pi_m^1$ over $\langle V_{\kappa}, \in , S \rangle$.  And let us prove the result  for  $n$. $S \subseteq \pk{\kappa}$ is $n$-s-stationary in $\pk{\kappa}$ if and only if $\varphi_{n} (S)$, where 
\begin{alignat*}{2} 
\varphi_{n}(S)& \ : \ & \forall y_1 \forall y_1 \bigwedge_{m<n} \big[ ( y_1,y_2 \subseteq \pk{\kappa}\wedge \varphi_{m}(y_1) \wedge \varphi_{m}(y_2))  \rightarrow  \\ & & \exists z ( z \in S \wedge Reg(|z \cap \kappa|) \wedge \varphi_m'(y_1,z) \wedge \varphi_m'(y_2,z) )\big]. 
\end{alignat*}
Notice that $y_1,y_2$ range over all possible subsets of $\pk{\kappa}$, this is, over elements of $\mathcal{P}(\pk{\kappa})\subseteq V_{\kappa+1}$. Then $y_1,y_2$ are second-order variables over $\langle V_{\kappa}, \in , S \rangle$. By induction hypothesis of (1), we have that $ \varphi_{m}(y_1)$ and $\varphi_{m}(y_2)$ are $\Pi_m^1$ over $\langle V_{\kappa}, \in , S \rangle$. And by induction hypothesis of (2), $ \varphi_m'(y_1,z)$ and $\varphi_m'(y_2,z)$ are $\Pi_0^1$ over $\langle V_{\kappa}, \in \rangle$. Hence, $\varphi_n(S)$ is  $\Pi_{(n-1)+1}^1=\Pi_{n}^1$ over $\langle V_{\kappa}, \in , S \rangle$.

 \quad\\
 (2) Let $x \in \pk{\kappa}$ be such that $|x \cap \kappa|$ is a regular limit cardinal. And suppose that for all $m< n$, $\varphi'_{m} (T,x ):``T \cap  \pkld{|x \cap \kappa|}{x}$ is $m$-s-stationary in $\pkld{|x \cap \kappa|}{x}"$  is a $\Pi_0^1$ statement over $\langle V_{\kappa}, \in\rangle$ in the parameters $T$ and $x$. And let us prove the result for $n$. $T \cap  \pkld{|x \cap \kappa|}{x}$ is $n$-s-stationary in $\pkld{|x \cap \kappa|}{x}$ if and only if $\varphi_n'(T , x )$, where  
\begin{alignat*}{2} 
\varphi_n' (T,x)& \ : \ & \forall y_1,\forall y_2 \bigwedge_{m<n}  \big[ ( y_1,y_2 \subseteq  \pkld{|x \cap \kappa|}{x} \wedge \varphi_{m}'(y_1,  x )\wedge \varphi_{m}'(y_2,  x ) )  \rightarrow \\ & & \exists z  ( z \in T \wedge Reg(|z \cap \kappa|) \wedge \varphi_m'(y_1,z) \wedge \varphi_m'(y_2,z)) \big]. 
\end{alignat*}
Notice that $y_1,y_2$ range over all possible subsets of  $ \pkld {|x \cap \kappa|}{x} \in V_{\kappa}$. Then $y_1,y_2$ are  first-order variables over $\langle V_{\kappa}, \in \rangle$. Now, by induction hypothesis of (2), we have that $\varphi_m'(y_1,x), \varphi_m'(y_2,x)),\varphi_m'(y_1,z)$ and $ \varphi_m'(y_2,z)$  are $\Pi_m^1$ statements over $\langle V_{\kappa}, \in , S \rangle$, in their respective parameters. Then $ \varphi_n' (T,x)$ is also a $\Pi_0^1$ statement over $\langle V_{\kappa}, \in \rangle$, in the  parameters $T$ and $x$. 
\end{proof}\vspace{1mm} 

\begin{corollary}\label{stp1stf}  
Let $\kappa$ be a regular limit cardinal. Then 
\begin{itemize}
\item[(1)] The formula $\psi_n(S)$ : $``S \subseteq \pk{\kappa}$ is $n$-stationary in $\pk{\kappa}"$ is  $\Pi^1_n $   over $\langle V_{\kappa}, \in , S \rangle$. \item[(2)] Let $x \in  \pk{\kappa}$ be such that $|x \cap \kappa|$ is a regular limit cardinal. Then  $\psi'_n(T,x) : ``T  \cap \pkld{|x \cap \kappa|}{x}$ is $n$-stationary in $\pkld{|x \cap \kappa|}{x}"$ is a $\Pi^1_0$ statement over $\langle V_{\kappa}, \in  \rangle$, in the parameters $T,x$.
\end{itemize}
\end{corollary}

\begin{proof} 
Analogous to the proof of Proposition \ref{p320bg}. Notice that,  $\psi_0(S)=\varphi_0(S)$ and $\psi'_0(T,x)=\varphi'_0(T,x)$. For $n \geq 1$, the formulas differ syntactically in only one aspect;  $\psi(S)$ and $\psi'_n(T)$ quantifies over subsets of $\pk{A}$, while $\varphi(S)$ and $\varphi'_n(T,x)$ quantifies over pairs of subsets of $\pk{A}$. 
\end{proof}\vspace{1mm} 

Recall that $x \in \pk{\kappa}$ if and only if  $\langle V_{\kappa}, \in \rangle \models \exists \alpha ( OR(\alpha) \wedge  x \subseteq \alpha )$. Then, in $\langle V_{\kappa}, \in \rangle$, the formulas expressing $x \in \pk{\kappa}$ and $x \subseteq \pk{\kappa}$ does not depend on $\kappa$. Thus, in Proposition \ref{p320bg} and Corollary \ref{stp1stf}, the formulas $\varphi_n(S)$ and $\psi_n(S)$ do not have $\kappa$ as a parameter. Hence, for every $\kappa$ regular limit cardinal, we can use uniformly that  $\varphi_n(S)$ and $\psi_n(S)$ are  $\Pi^1_n $   over $\langle V_{\kappa}, \in , S \rangle$.\vspace{2mm}

 \begin{theorem}\label{inds} 
 Let $n < \omega$. If $\kappa$ is  $\Pi_n^1$ indescribable, then $\pk{\kappa}$ is both, $(n+1)$-s-stationary and $(n+1)$-stationary in  $\pk{\kappa}$.
 \end{theorem}

\begin{proof}  
Suppose $\kappa$ is   $\Pi_n^1$ indescribable. Let $S$ be $m$-s-stationary in $\pk{\kappa}$ for some  $m < n+1$. By Proposition \ref{p320bg},  the sentence $ \varphi_m(S):``S \subseteq \pk{\kappa}$ is $m$-s-stationary in $\pk{\kappa}"$ is  $\Pi_m^1$  over $\langle V_{\kappa}, \in , S \rangle $. We have that,
 $$\langle V_{\kappa}, \in , S \rangle \models \varphi_m(S).$$
 
Then by $\Pi_n^1$ indescribability of $\kappa$, there is a regular limit cardinal   $\mu < \kappa$ with,
 \begin{align*}
 \langle V_{\mu}, \in , S\cap V_{\mu} \rangle \models \varphi_m(S\cap V_{\mu}).
 \end{align*} 

Notice that $\pkld{\mu}{\mu} = \pk{\kappa} \cap V_{\mu}$. Direct inclusion is immediate. For the converse, take $x \in \pk{\kappa} \cap V_{\mu}$. Then $x \subseteq \kappa \cap V_{\mu}=\mu$ and $|x| < \mu$, otherwise, $\text{rank}(x) = \mu $ and so $x \notin V_{\mu}$. Then $ x \in \pkld{\mu}{\mu}$, and so $\pk{\kappa} \cap V_{\mu} \subseteq \pkld{\mu}{\mu} $.  \vspace{2mm}

Now,  $S \cap V_{\mu}=(S \cap \pk{\kappa}) \cap V_{\mu} = S \cap  \pkld{\mu}{\mu}$. Then we have  $\langle V_{\mu}, \in , S\cap \pkld{\mu}{\mu}\rangle \models \varphi_m(S \cap \pkld{\mu}{\mu})$. Thus, $\mu \in \pk{\kappa}$ is such that $\mu \cap \kappa = \mu$ is a regular limit cardinal, and $S\cap \pkld{\mu}{\mu}$ is $m$-s-stationary in $ \pkld{\mu}{\mu}$. Therefore, $ \pk{\kappa}$ is $(n+1)$-s-stationary in  $\pk{\kappa}$. To prove that $ \pk{\kappa}$ is $(n+1)$-stationary in  $\pk{\kappa}$, we proceed analogously, using the formula $\psi_n(S)$, which,  by Corollary \ref{stp1stf},  is  $\Pi_m^1$  over $\langle V_{\kappa}, \in , S \rangle $. 
\end{proof}\vspace{1mm} 

 \begin{corollary}\label{totind}  
 Let $\kappa$ be totally indescribable. Then $\pk{\kappa}$ is both $n$-s-stationary and $n$-stationary in  $\pk{\kappa}$, for every $n <\omega$. \ $\square$
 \end{corollary}

 \begin{theorem}\label{lsc} 
 If $\kappa$ is $\lambda$-supercompact  and $\lambda ^{< \kappa } = \lambda$, then $\pkl$ is $n$-s-stationary, for every $n<\omega$.
 \end{theorem}

\begin{proof} 
First, notice that, since  $\kappa$ is $\lambda$-supercompact, there is an elementary embedding $j: V \preceq M$ such that $\text{crit}(j)= \kappa$, $\lambda < j(\kappa)$ and $^{\lambda}M \subseteq M$, where $M$ is transitive. Notice that $j``\alpha =\{j(\beta) : \beta \in \alpha \}\in M$ for all $\alpha \leq \lambda$, because $j(\beta)\in M$ for all $\beta <\alpha \leq \lambda$ and $M$ is closed under $\lambda$ sequences. Then, since $\kappa \leq \lambda$, we have,  $\pkld{\kappa}{j``\lambda} \subseteq M$. Moreover, as $|j``\lambda|=|\lambda|$, it follows that  $|j``\lambda| ^{< \kappa} =\lambda ^{< \kappa } = \lambda $, and so $|\pkld{\kappa}{j``\lambda} |=\lambda$. Therefore, $\pkld{\kappa}{j``\lambda} \in  M$ and $\mathcal{P}(\pkld{\kappa}{j``\lambda} ) \subseteq M $.\vspace{2mm}

Also, observe that $j {\restriction} {\kappa}= id_{\kappa}$. Then $g : \lambda \rightarrow j``\lambda$, given by $g(\alpha):=j(\alpha)$, is a bijection such that $g(\alpha)=\alpha$ for all $\alpha \in c(\kappa)$. Then by Lemma \ref{funt}, the induced isomorphism $f:=g^*:\langle \pkl , \subseteq \rangle \rightarrow \langle  \pk{j''\lambda} , \subseteq \rangle$  given by $f(x):=g[x]=j''x$, preserves $n$-s-stationary sets for all $n <\omega$. Moreover, if $x \in \pkl$ then $|x|<\kappa=crit(j)$, and so $j''x=j(x)$. Thus, for all $x \in \pkl$,  $f(x)=j''x=j(x)$.\vspace{2mm}

Now, let us proceed to prove the Theorem. Fix $n< \omega$. Take $m<n $ and  $S_1,S_2 \subseteq \pkl$  $m$-s-stationary sets in $\pkl$. As $f$ preserves $n$-s-stationary sets, $f[S_1],f[S_2]\subseteq \pk{j``\lambda}$ are $m$-s-stationary in $\pk{j``\lambda}$. Notice that, for $i=1,2$, $f[S_i]=\{f(x) : x \in S_i\}=\{j(x) : x \in S_i\}=j``S_i\subseteq j(S_i)$. Then, we have that $j(S_1)\cap \pk{j``\lambda}$ and $j(S_1)\cap \pk{j``\lambda}$ are also $m$-s-stationary in $\pk{j``\lambda}$. This is,\vspace{1mm}
 \begin{eqnarray*}
 V & \models &`` \ j(S_1)\cap \pk{j``\lambda} \text{ and } j(S_2)\cap \pk{j``\lambda} \\ & &  \text{ are $m$-s-stationary in } \pk{j``\lambda} \ " 
 \end{eqnarray*}
 
Notice that the formula $``X$ is $m$-s-stationary in  $\pkld{\kappa}{j``\lambda}"$, depends only on $X$, $\kappa$ and the subsets of $\pkld{\kappa}{j``\lambda}$. Now, $M$ is transitive, $j(S_1) \cap  \pk{j``\lambda} \in M$, $j(S_2) \cap  \pk{j``\lambda}\in  M$ and $\mathcal{P}(\pkld{\kappa}{j``\lambda} ) \subseteq M $, then, \vspace{1mm}
 \begin{eqnarray*}
 M & \models &`` \ j(S_1)\cap \pk{j``\lambda} \text{ and } j(S_2)\cap \pk{j``\lambda} \\ & &  \text{ are $m$-s-stationary in } \pk{j``\lambda} \ " \end{eqnarray*}

In $M$, $\kappa$ is a regular limit cardinal and $\kappa < j(\kappa)$.  Define $x:=j``\lambda$. Then $\kappa=j``\kappa \subseteq j``\lambda =x$, and so $\kappa \subseteq x \cap j(\kappa)$. In fact,  $\kappa = x \cap j(\kappa)$. For if $\alpha \in (x \cap j(\kappa))\setminus \kappa$, then $\alpha =j(\beta)$ for some $\kappa < \beta < \lambda$ and $\alpha < j(\kappa)$, but $\kappa < \beta $ implies $j(\kappa) < j(\beta)= \alpha$,  this is a contradiction. Now, $x=j``\lambda \subseteq j(\lambda)$ and $|x|=|j``\lambda|= \lambda <j(\kappa)$. Then $x \in \pkld{j(\kappa)}{j(\lambda)}$, and so \vspace{1mm}
 \begin{eqnarray*}
 M & \models & \exists x \big(Reg(|x \cap j(\kappa)|)  \wedge x \in \pkld{j(\kappa)}{j(\lambda)}  \wedge ``j(S_1) \cap  \pkld{|x \cap j(\kappa)| }{x} \\ & & \text{ and } j(S_2) \cap  \pkld{|x \cap j(\kappa)| }{x}  \text{ are } m\text{-s-stationary in } \pkld{|x \cap j(\kappa)| }{x}"\big).
\end{eqnarray*}
 As $j$ is an elementary embedding, we get that \vspace{1mm}
 \begin{eqnarray*}
 V & \models & \exists x \big(Reg(|x \cap j^{-1}(j(\kappa))|) \ \wedge  \ x \in \pkld{j^{-1}(j(\kappa))}{j^{-1}(j(\lambda))} \ \wedge  \\ & &
 ``j^{-1}(j(S_i)) \cap  \pkld{|x \cap j^{-1}(j(\kappa))|}{x}  \text{ and } j^{-1}(j(S_i)) \cap  \pkld{|x \cap j^{-1}(j(\kappa))|}{x}  \\ & & \text{ are } m\text{-s-stationary in } \pkld{|x \cap j^{-1}(j(\kappa))|}{x}"\big).
 \end{eqnarray*}
Since $ j^{-1}(j(\kappa)) = \kappa$, $ j^{-1}(j(\lambda))= \lambda$  and $ j^{-1}(j(S_i))=S_i$,  we   have \vspace{1mm}
\begin{eqnarray*}
V & \models & \exists x \big(Reg(|x \cap \kappa|)\ \wedge \ x \in \pkl \ \wedge``S_1 \cap  \pkld{|x \cap \kappa|}{x} \text{ and }  \\ & &  S_i \cap  \pkld{|x \cap \kappa|}{x} \text{ are } m\text{-s-stationary in } \pkld{|x \cap \kappa|}{x}" \big).
 \end{eqnarray*}
 
Therefore, for $i=1,2$, there exists $x \in \pkl$ such that  $\mu := |x \cap \kappa|$ is a regular limit cardinal and $S_1\cap \pkld{\mu}{x}$, $S_2\cap \pkld{\mu}{x}$ are $m$-s-stationary in $\pkld{\mu}{x}$. Therefore, $\pkl$ is $n$-s-stationary in $\pkl$. 
\end{proof} 

\begin{corollary}\label{lsc2} 
If $\kappa$ is $\lambda$-supercompact  and $\lambda ^{< \kappa } = \lambda$, then $\pkl$ is $n$-stationary, for every $n<\omega$.
\end{corollary}

\begin{proof} 
Analogous to the proof of Theorem \ref{lsc}. Noticing that, by Lemma \ref{funt2}, the induced isomorphism $f:\langle \pkl , \subseteq \rangle \rightarrow \langle  \pk{j''\lambda} , \subseteq \rangle$  also preserves $n$-stationary sets for all $n <\omega$. And taking throughout the proof only one $S \subseteq \pkl$  that is $m$-stationary  in $\pkl$.
\end{proof} 

\section*{Acknowledgments}

In compiling this research paper, I am grateful for the insightful discussions and invaluable guidance provided by my PhD advisor, Professor Joan Bagaria. Additionally, I extend my thanks to Professors Sakaé Fuchino and Hiroshi Sakai for engaging and enlightening discussions during my research stay at Kobe. Their perspectives and input have been instrumental in shaping the ideas presented in this paper.

\bibliographystyle{amsplain} 
\bibliography{masterbiblio}

\end{document}